\documentclass[11pt]{amsart}
\usepackage[T1]{fontenc}
\usepackage{lmodern,amsmath,amssymb,amsthm,mathtools,booktabs}
\usepackage[margin=1in]{geometry}
\usepackage{xcolor,enumitem}
\usepackage[colorlinks=true,linkcolor=blue!45!black,citecolor=blue!45!black,urlcolor=blue!45!black]{hyperref}
\usepackage{microtype}
\usepackage{tikz,tikz-cd}
\newtheorem{theorem}{Theorem}[section]
\newtheorem{proposition}[theorem]{Proposition}
\newtheorem{lemma}[theorem]{Lemma}
\newtheorem{corollary}[theorem]{Corollary}
\newtheorem{conjecture}[theorem]{Conjecture}
\newtheorem{problem}{Problem}
\theoremstyle{remark}
\newtheorem{remark}[theorem]{Remark}
\newtheorem{example}[theorem]{Example}
\newcommand{\Z}{\mathbb Z}
\newcommand{\htq}{\operatorname{ht}}
\newcommand{\mword}[1]{[#1]}
\newcommand{\corder}[1]{\langle #1\rangle}
\newcommand{\qtuple}[1]{\mathbf b(#1)}
\newcommand{\oqtuple}[2]{\mathbf b_{#1}(#2)}
\newcommand{\file}[1]{\texttt{#1}}
\newcommand{\claim}[1]{\textup{[\texttt{#1}]}}

\title[Mutation-acyclicity in rank four]{Mutation-acyclicity in rank four is decidable but not semialgebraic}

\author{Kymani T. K. Armstrong-Williams}
\address{Centre for Theoretical Physics and Astronomy, Queen Mary University of London, London E1 4NS, United Kingdom}
\email{k.t.k.armstrong-williams@qmul.ac.uk}
\author{Edward Hirst}
\address{Institute of Mathematics, Statistics and Scientific Computing, University of Campinas (Unicamp), Campinas 13083-859, Brazil}
\email{ehirst@unicamp.br}
\author{Blake Jackson}
\address{Institute for Computer-Aided Reasoning in Mathematics, Carnegie Mellon University, 4720 Forbes Avenue, Pittsburgh, PA 15213, USA}
\email{jackson@icarm.io}
\author{Kyu-Hwan Lee}
\address{Department of Mathematics, University of Connecticut, Storrs, CT 06269, USA; Korea Institute for Advanced Study, Seoul 02455, Republic of Korea}
\email{khlee@math.uconn.edu}

\thanks{
K.T.K.A.W.\ was supported by a Royal Society Career Development Fellowship. E.H.\ was supported by a
S\~ao Paulo Research Foundation (FAPESP) grant 2024/18994-7.
B.J.\ was supported by ICARM through NSF grant DMS-2425401. K.-H.L.\ was supported by NSF grant DMS-2617278.}

\begin{document}
\begin{abstract}
    We prove that mutation-acyclicity of quivers on four vertices is decidable.
    If a four-vertex quiver $Q$ is mutation-acyclic, then some sequence of mutations takes it to an acyclic quiver without any arrow multiplicity exceeding $\max\{\htq(Q),m(Q)\}$, where $\htq(Q)$ is the largest arrow multiplicity of $Q$ and $m(Q)$ is the Markov invariant of Fomin and Neville evaluated on the totally proper wiggle class of $Q$.
    A search within a computable form of this bound decides mutation-acyclicity of four-vertex quivers with $O(h^{24})$ states for input height $h\ge1$.
    Together with a theorem of Caldero and Keller, it also decides mutation equivalence of two four-vertex quivers up to relabeling whenever one of them is mutation-acyclic.
    We conjecture that $\htq(Q)$ alone suffices as the bound, which would reduce the search to $O(h^{6})$ states, and we verify this for all four-vertex quivers of height at most three.
    In the other direction, for every $n\ge4$ we construct an affine family of connected $n$-vertex quivers whose nonnegative even members are mutation-acyclic and whose positive odd members are mutation-cyclic.
    Consequently, unlike in rank three, no fixed finite Boolean combination of polynomial equations and strict polynomial inequalities in the arrow weights (i.e., a semialgebraic criterion) decides mutation-acyclicity of $n$-vertex quivers, even allowing polynomials that are not mutation-invariant.
    Finally, we correct the criterion for four-vertex quivers of height at most two stated in our earlier paper, and recover its classification counts with short certificates.
    The calculations in the proof of the height bound are verified by computer for all values of their parameters.
\end{abstract}
\maketitle

\section{Introduction}
Quivers are finite directed multigraphs without loops or directed 2-cycles.
Fomin and Zelevinsky introduced quiver mutation in their theory of cluster algebras \cite{FZ1,FZ2}.
A quiver is \textbf{mutation-acyclic} if some sequence of mutations takes it to a quiver without directed cycles, and \textbf{mutation-cyclic} otherwise.
Mutation-acyclic quivers are the quivers of acyclic cluster algebras, which have been studied through cluster categories and quiver representations \cite{BMR,BMRT2007,caldero2006,caldero2008}, through $c$-vectors \cite{speyer-thomas2013,seven2015}, and through reflection groups \cite{barot-marsh2015,felikson-tumarkin2016,felikson-tumarkin2018}.
This paper concerns the following decision problem.
\begin{problem}\label{q:mutation-acyclicity}
    Given a quiver $Q$, decide whether $Q$ is mutation-acyclic.
\end{problem}
Neville records that no algorithm for Problem~\ref{q:mutation-acyclicity} is known \cite[Chapter~I]{neville-thesis}, and in our earlier paper we noted that no algorithm or mutation invariant was known to decide mutation-acyclicity in rank four or higher \cite{AHJL2025}.
The problem is closely related to deciding mutation equivalence, for which Fomin, Williams and Zelevinsky ask for an algorithm \cite[Problem~2.6.12]{FWZ}, and for which no algorithm is known \cite{fomin-neville2026}, \cite[Remark~2.3]{FIL2021}, \cite[slide~7]{fomin-opac}.

For quivers with at most three vertices, Problem~\ref{q:mutation-acyclicity} is solved.
Quivers with one or two vertices are acyclic.
For three vertices, Assem, Blais, Br\"ustle and Samson gave a descent algorithm \cite[Section~2.5]{assem2006}, and Beineke, Br\"ustle and Hille proved that an oriented triangle is mutation-cyclic exactly when its multiplicities are at least two and its Markov constant is at most four \cite[Theorem~1.2]{beineke2011}.
Seven extended the description of rank-three mutation classes to skew-symmetrizable matrices \cite{seven2013}, Akagi gave the skew-symmetrizable form of the Markov-constant criterion \cite{akagi2024}, and Felikson and Tumarkin and Casals and Ke studied rank-three quivers with real weights \cite{felikson-tumarkin2019,casals-ke2026}.
Quivers of finite mutation type, classified by Felikson, Shapiro and Tumarkin \cite{FST2012} (see also \cite{derksen-owen2008} and, for acyclic quivers, \cite{buan-reiten2006}), have finite mutation classes, so for them the problem reduces to a finite enumeration.

For quivers with four or more vertices, only partial results are known.
Felikson and Tumarkin showed that a quiver mutation-equivalent to an acyclic quiver with at least two arrows between every pair of vertices reaches such a quiver by a descent that never increases an arrow multiplicity \cite[Corollary~3.5 and Remark~3.6]{felikson-tumarkin2018}, and they note that Corollary~3.5 is also proved in Warkentin's thesis \cite{warkentin2014}.
Lawson and Mills determined which minimal mutation-infinite quivers are mutation-acyclic \cite[Section~4]{lawson-mills2018}.
In an earlier paper, we classified the four-vertex quivers with at most two arrows between any two vertices \cite{AHJL2025}, in the course of a machine-learning study; see \cite{BFHHMX2020,HJCBRBK2025, Dechant:2022ccf} for other applications of machine learning to quiver mutation.
The criterion stated in \cite[Theorem~2.9]{AHJL2025} is wrong, and we correct it in Section~\ref{sec:correction}.

There is also a long list of necessary conditions for mutation-acyclicity.
Buan, Marsh and Reiten proved that every full subquiver of a mutation-acyclic quiver is mutation-acyclic \cite[Corollary~5.3]{BMR}, a consequence of their Theorem~2.13 that was pointed out to them by A.~Seven, and Warkentin gave a more explicit proof \cite{warkentin2011}.
Both proofs are representation-theoretic, and an elementary proof is still wanted \cite[Example~4.1.5]{FWZ45}, \cite[slides~11--12]{fomin-opac}.
Every mutation-acyclic quiver admits an admissible quasi-Cartan companion \cite[Theorem~1.2]{seven2015}, a notion going back to Barot, Geiss and Zelevinsky \cite{barot2006}.
Every mutation-acyclic quiver also has a reddening sequence, since acyclic quivers have maximal green sequences \cite[Lemma~2.20]{BDP2014} and the existence of a reddening sequence is mutation invariant \cite[Corollary~19]{muller2016}.
Neville proved that every mutation-acyclic quiver is totally proper \cite[Theorem~1.1]{neville2025}, as conjectured by Fomin and Neville \cite[Conjecture~15.9]{fomin-neville2026}.
Warkentin's forks give further mutation-invariant tools \cite{warkentin2014}, developed by Ervin \cite{ervin2024}.
In rank four there are mutation-cyclic quivers without mutation-cyclic full subquivers on three vertices, such as the dreaded torus \cite[Section~7.2]{matherne-muller2015}, whose cluster algebra comes from a torus with one boundary component and one marked point \cite{FST2008}, and Warkentin's oriented four-cycles with multiplicities $a,b,a,b\ge2$ \cite[Example~6.11]{warkentin2014}.
Mutation invariants beyond the determinant are studied in \cite{chelnokov2023,casals2024,seven-unal2025,ervin2024,fomin-neville2025,fomin-neville2026}, and Ervin and Jackson ask for a description of mutation-acyclicity as an intersection of hereditary mutation-invariant properties \cite{ervin-jackson-topology}.

The starting point for our main result is a finiteness theorem of Neville.
Fomin and Neville attach to a quiver with a cyclic ordering of its vertices a Markov invariant \cite{fomin-neville2026}.
For a mutation-acyclic quiver $Q$, write $m(Q)$ for its value on the totally proper wiggle class of $Q$ (Section~\ref{sec:background}), and $\htq(Q)$ for the largest arrow multiplicity of $Q$.
Neville proved that the arrow multiplicities of every acyclic quiver in the mutation class of $Q$ are bounded in terms of $m(Q)$, so that only finitely many acyclic quivers share a given Markov invariant \cite[Proposition~3.21]{neville2025}.
Our main result extends this bound from the acyclic quivers of a mutation class to a mutation sequence leading to one of them.

\begin{theorem}\label{thm:introheight}
    Let $Q$ be a mutation-acyclic quiver on four vertices.
    Then some sequence of mutations takes $Q$ to an acyclic quiver, with every quiver along the sequence of height at most $\max\{\htq(Q),m(Q)\}$.
    Consequently, mutation-acyclicity of four-vertex quivers is decidable.
    For input height $h\ge1$, the decision procedure searches $O(h^{24})$ quivers.
\end{theorem}
The height statement is Theorem~\ref{thm:height}, and the decision procedure and the state count are Corollaries~\ref{thm:algorithm} and~\ref{cor:complexity}.
Since the totally proper class of $Q$ is not known in advance, the procedure uses the largest Markov value over all proper wiggle classes of $Q$ in place of $m(Q)$, and a four-vertex quiver without proper wiggle classes is a vortex and hence mutation-cyclic.
The bound concerns the existence of one suitable mutation sequence, and other routes to an acyclic quiver may climb arbitrarily high.
The running time is pseudo-polynomial in the arrow multiplicities, rather than polynomial in their binary encoding length.

The proof chooses a route to an acyclic quiver that minimizes its maximum height, then the number of quivers at that height, then its length, and shows that a stretch of such a route above $\max\{\htq(Q),m(Q)\}$ could be replaced by a lower one.
The replacements are found by a case analysis on how the first two mutation vertices of such a stretch are joined.
Every calculation in this case analysis is verified by a computer program, exactly and for all values of the parameters (Appendices~\ref{app:lemmas} and~\ref{app:signs}).
Descent methods face an obstacle in rank four.
Fomin and Neville constructed arbitrarily long primitive mutation cycles of four-vertex quivers, and observed that any descent algorithm would have to be confluent over arbitrarily large diamonds \cite{fomin-neville2025}; see also \cite{ervin-neville2026} and \cite[Chapter~I]{neville-thesis}.
Soukup observes that if mutation equivalence were undecidable, some equivalent quivers would be joined only by extremely long mutation sequences, going far ``into the sky'', and suggests that a bound like that of Coward and Lackenby for Reidemeister moves may exist for quiver mutation \cite[Sections~4.1--4.2]{soukup2025}, \cite{coward-lackenby2014}.
Theorem~\ref{thm:introheight} gives such a bound for mutation-acyclicity in rank four, since a bound on height also bounds the length of a shortest route.

We do not know whether the term $m(Q)$ is needed.
We conjecture that $\htq(Q)$ alone suffices, which would reduce the search to $O(h^6)$ quivers, and we verify this for all four-vertex quivers of height at most three (Conjecture~\ref{conj:height}).
The height bound also certifies mutation-cyclicity in an example where none of the other obstructions of this paper applies to any quiver within the search range (Example~\ref{ex:height3}).
Combined with the description of acyclic seeds by Caldero and Keller \cite[Corollary~4 and its proof]{caldero2006}, Theorem~\ref{thm:introheight} also decides mutation equivalence of two four-vertex quivers up to relabeling whenever at least one of them is mutation-acyclic (Corollary~\ref{cor:equivalence}).

Our second main result shows that no criterion of the kind that works in rank three can exist in higher rank.
The rank-three criterion of Beineke, Br\"ustle and Hille is a finite Boolean combination of polynomial sign conditions on the arrow weights.
\begin{theorem}\label{thm:nogo}
    Let $n\ge4$.
    No fixed, finite Boolean combination of polynomial equations and strict polynomial inequalities in the $\binom n2$ arrow weights, with real coefficients and not necessarily mutation-invariant, decides mutation-acyclicity for all integral quivers on $n$ vertices.
    This remains true when the quivers are required to be connected.
\end{theorem}
For $n=4$ the proof uses the affine family $Q(t)$, $t\in\Z$, of Section~\ref{sec:family}, with arrow weights $(-2,t+2,t,-t,2-t,0)$ in the coordinates of Section~\ref{sec:background}.
The family is mutation-acyclic exactly when $t<0$ or $t$ is even, so for large positive $t$ its mutation-acyclicity changes at every successive integer.
In contrast, substituting the six affine arrow weights into any fixed polynomial gives a one-variable polynomial, whose sign is eventually constant unless it vanishes identically.
For $n>4$ we attach a path to $Q(t)$.

Chelnokov proved that a mutation invariant of four-vertex quivers that is polynomial on each sign pattern of the arrow weights is a polynomial in the determinant \cite[Theorem~1]{chelnokov2023}.
The family $Q(t)$ has determinant $16$ for every $t$, so criteria built from such invariants cannot decide mutation-acyclicity, and Theorem~\ref{thm:nogo} complements this by removing the invariance hypothesis.
The family is also invisible to the Markov invariant and to properness, since every $Q(t)$ with $t\ge0$ has a single proper wiggle class, with the same Markov value as the acyclic quiver $Q(0)$ (Remark~\ref{rem:familymarkov}).
Theorem~\ref{thm:nogo} does not rule out arithmetic tests, mutation invariants in general, or algorithms whose number of steps depends on the input, and Theorem~\ref{thm:introheight} is an algorithm of the last kind.

Finally, Section~\ref{sec:correction} corrects the criterion of \cite[Theorem~2.9]{AHJL2025} for four-vertex quivers of height at most two and recovers the classification counts with short certificates.
All computations in this paper are carried out by a supplementary Python package, and checking them needs only the Python standard library.
The program \file{check\_proofs.py} verifies every calculation delegated from the appendices, \file{check\_census.py} verifies the certificates of Section~\ref{sec:correction} with an independent implementation of mutation, and \file{examples.py} checks the examples.
The main results are also formalized in Lean~4 with mathlib, and the Lean sources are available with the supplement at \url{https://github.com/KTKAW/rank4-paper}.
The formalization proves decidability of mutation-acyclicity in rank four with a coarser computable height bound than that of Theorem~\ref{thm:introheight}, the case $n=4$ of Theorem~\ref{thm:nogo} for rational coefficients, Theorems~\ref{thm:parity} and~\ref{thm:census}, and Proposition~\ref{prop:box}, the last for every oriented four-cycle with multiplicities $a,b,a,b\ge2$.

\section{Quivers and mutation}\label{sec:background}

A \textbf{quiver} of rank $n$ is a quiver on the vertices $1,\dots,n$.
Unless we explicitly take a smaller subquiver or say otherwise, our quivers have rank four.
Every rank-$n$ quiver $Q$ has a skew-symmetric integral \textbf{exchange matrix}
$B_Q=(b_{ij}(Q))$, where $b_{ij}(Q)>0$ means that there are $b_{ij}(Q)$ arrows $i\to j$.
For rank four, we encode $Q$ by the tuple
\begin{equation}\label{eq:matrix}
    \qtuple{Q}=(b_{12}(Q),b_{13}(Q),b_{14}(Q),b_{23}(Q),b_{24}(Q),b_{34}(Q))=(a,b,c,d,e,f),
\end{equation}
so that
\[
    B_Q=\begin{pmatrix}0&a&b&c\\-a&0&d&e\\-b&-d&0&f\\-c&-e&-f&0\end{pmatrix}.
\]
The coordinates of every six-tuple in this paper are in the fixed order
$12,13,14,23,24,34$, unless a cyclic ordering is explicitly attached as a subscript.
When a quiver is fixed in a calculation, we may suppress its argument and write $b_{ij}$.
Six-entry parenthesized lists denote coordinate tuples; mutation words and cyclic orderings use different brackets, as specified below.
We say a quiver is \textbf{acyclic} if it contains no directed cycles of any length.

The operation of interest is \textbf{mutation} at a vertex $k$.
To mutate at vertex $k$, first add an arrow $i \to j$ for every two-path $i\to k\to j$, counting multiplicities.
Next, reverse all arrows touching $k$, and cancel pairs of arrows pointing in opposite directions.
The resulting quiver is denoted $\mu_kQ$.
In matrix form \cite{FZ1}, abbreviate $b_{ij}=b_{ij}(Q)$ and $b'_{ij}=b_{ij}(\mu_kQ)$. Then
\begin{equation}\label{eq:mutation}
b'_{ij}=\begin{cases}
    -b_{ij},&i=k\text{ or }j=k,\\
    b_{ij}+[b_{ik}]_+[b_{kj}]_+-[-b_{ik}]_+[-b_{kj}]_+,&i,j\ne k,
\end{cases}\qquad [x]_+=\max(x,0).
\end{equation}
Mutation is an involution, and preserves the vertex sets of the connected components.
A quiver is \textbf{mutation-acyclic} if a finite sequence of mutations takes it to an acyclic quiver; otherwise it is \textbf{mutation-cyclic}.
If a quiver is acyclic, a \textbf{topological ordering} is a total order on its vertices such that every directed path of the quiver points from an earlier vertex to a later vertex.
In other words, if there is a directed path $i \to j$ in the quiver, then $i < j$ in the topological ordering.

We write a \textbf{mutation word} as a bold letter $\mathbf w=\mword{k_1,\dots,k_r}$.
The square brackets list the vertices in the order of mutation, from left to right:
\[
    \mu_{\mathbf w}Q=\mu_{k_r}\bigl(\cdots\mu_{k_2}(\mu_{k_1}Q)\cdots\bigr).
\]
For example, $\mu_{\mword{1,2,3}}Q=\mu_3(\mu_2(\mu_1Q))$.
The empty word is $\mword{}$, and juxtaposition of words means concatenation in the order performed.
A mutation sequence is \textbf{reduced} if consecutive mutation vertices differ; since mutation is an involution, mutations at the same vertex back-to-back cancel each other out.

The \textbf{height} of a quiver $Q$ is
\[
    \htq(Q)=\max_{i<j}|b_{ij}(Q)|.
\]
A \textbf{full subquiver} of a quiver $Q$ consists of a subset of vertices and all arrows between them.
On the level of the exchange matrix, the full subquiver on a subset of vertices is obtained by taking the full submatrix on the rows and columns corresponding to the subset of vertices.

For a permutation $\pi\in S_n$, write $Q^\pi$ for the \textbf{relabeling} defined by
\[
    b_{ij}(Q^\pi)=b_{\pi(i)\pi(j)}(Q).
\]
Thus old vertex $\pi(i)$ receives new label $i$.
We write permutations in cycle notation, with spaces between successive vertices and no commas.
The two relabelings used repeatedly in rank four are
\[
    \pi_{12}=(1\,2),\qquad \pi_{123}=(1\,2\,3),
\]
with all unlisted vertices fixed. In particular, $\pi_{123}(1)=2$, $\pi_{123}(2)=3$, and $\pi_{123}(3)=1$.
Write $Q^{\mathrm{op}}$ for the quiver obtained by reversing all arrows, so that $B_{Q^{\mathrm{op}}}=-B_Q$.
Relabeling and reversal of all arrows preserve height and acyclicity.
Both commute with mutation, in the following sense.
For reversal, \eqref{eq:mutation} gives $\mu_k(Q^{\mathrm{op}})=(\mu_kQ)^{\mathrm{op}}$ directly.
To transport a mutation through a relabeling, use
\begin{equation}\label{eq:tailtransport}
    \mu_{\pi^{-1}(k)}(Q^\pi)=(\mu_kQ)^\pi.
\end{equation}
For a word $\mathbf w=\mword{k_1,\dots,k_r}$, let
$\pi^{-1}\mathbf w=\mword{\pi^{-1}(k_1),\dots,\pi^{-1}(k_r)}$.
Then $\mu_{\pi^{-1}\mathbf w}(Q^\pi)=(\mu_{\mathbf w}Q)^\pi$.
Thus replacing an endpoint $Q$ by $Q^\pi$ requires replacing each later mutation vertex $k$ by $\pi^{-1}(k)$; all later heights and the acyclicity of the final quiver are preserved.

We now state some classical results of \cite[Theorem~1.2]{beineke2011} and \cite[Corollary~5.3]{BMR}, which we will use frequently enough to warrant restating here.
\begin{proposition}[{\cite[Theorem~1.2]{beineke2011}, \cite[Corollary~5.3]{BMR}}]\label{prop:rank3}
    A three-vertex quiver is mutation-cyclic if and only if it is an oriented triangle whose arrow multiplicities $x,y,z$ satisfy
    \[
        x,y,z\ge2,\qquad x^2+y^2+z^2-xyz\le4.
    \]
    The value $x^2+y^2+z^2-xyz$ is called the \textbf{Markov constant}, and the expression defining it is the \textbf{Markov polynomial}.
    Also, every full subquiver of a mutation-acyclic quiver is mutation-acyclic.
    This implies that a quiver is mutation-cyclic if it contains a mutation-cyclic full subquiver.
\end{proposition}
We call a full subquiver on three vertices a \textbf{full triangle}, even when some of its three pairs of vertices are not joined.
Finding a mutation-cyclic full triangle is a practical way to prove mutation-cyclicity for quivers of larger rank.

\subsection{Cyclic orderings}
We use the cyclic ordering framework of \cite{fomin-neville2026} throughout this paper.
A \textbf{cyclic ordering} puts the vertices of a quiver around a circle.
We write $\omega=\corder{v_1,\dots,v_n}$, with angle brackets identifying lists that differ by cyclic rotation, but not by reversal.
A linear ordering is written $v_1<\cdots<v_n$; wrapping it around the circle gives $\corder{v_1,\dots,v_n}$.
A vertex $k$ is \textbf{proper} if every directed path $i\to k\to j$ encounters $i,k,j$ in clockwise order according to the cyclic ordering.
In other words, every two-path through $k$ follows the circle clockwise.
A \textbf{wiggle} swaps two neighboring vertices on the circle that are not joined by an arrow.
Cyclic orderings related by wiggles are \textbf{wiggle equivalent}; we denote the class of $\omega$ by $[\omega]_{\mathrm{wig}}$.
A wiggle-equivalence class is itself proper if each vertex is proper in some ordering in the class.

To perform a \textbf{proper mutation} at $k$, first choose an ordering in which $k$ is proper.
Mutate the quiver and move $k$ clockwise past its incoming neighbors in $\mu_kQ$ (its outgoing neighbors in $Q$), without passing its outgoing neighbors in $\mu_kQ$.
The result is well defined up to wiggles.
A class is \textbf{totally proper} if it remains proper after every sequence of proper mutations \cite[Definition~14.1]{fomin-neville2026}.

\begin{theorem}[{\cite[Observation~11.6, Theorem~14.3, and Conjecture~15.9]{fomin-neville2026}, \cite[Theorem~1.1]{neville2025}}]\label{thm:FN}
    A quiver of any rank has at most one totally proper wiggle-equivalence class.
    Every mutation-acyclic quiver has such a class.
    For an acyclic quiver, it is the class obtained by wrapping a topological ordering around the circle.
\end{theorem}

\begin{remark}
    A practical consequence used throughout this paper is that if a quiver is known to be mutation-acyclic and has only one proper wiggle class, then that class is the totally proper one.
\end{remark}

\subsection{The Markov invariant}

For a cyclic ordering $\omega=\corder{v_1,v_2,v_3,v_4}$, choose the displayed linear representative and read the exchange matrix in that order:
\[
    B_\omega(Q)=(b_{v_i v_j}(Q))_{i,j=1}^4,\qquad
    \oqtuple{\omega}{Q}=(b_{v_1v_2}(Q),b_{v_1v_3}(Q),b_{v_1v_4}(Q),
    b_{v_2v_3}(Q),b_{v_2v_4}(Q),b_{v_3v_4}(Q)).
\]
The ordered matrix and tuple depend on the chosen linear representative; the Markov value below does not.
Writing $\oqtuple{\omega}{Q}=(a,b,c,d,e,f)$, define the \textbf{Markov invariant} by
\begin{equation}\label{eq:markov-invariant}
    M_\omega(Q)=a^2+b^2+c^2+d^2+e^2+f^2+abd+ace+bcf+def+acdf.
\end{equation}
This is the Markov invariant $M_Q$ of \cite[Equation (8.2)]{fomin-neville2026}, computed for the linear ordering $v_1<v_2<v_3<v_4$, and its explicit form for four vertices is \cite[Equation (9.6)]{fomin-neville2026}.
To match the two notations, note that the unipotent companion of \cite[Definition~4.2]{fomin-neville2026} has entries $-b_{v_i v_j}(Q)$ above the diagonal, so the letters $x,z,w,y,v,u$ of \cite[Examples~4.15 and~9.3]{fomin-neville2026} are our ordered coordinates $a,b,c,d,e,f$, respectively.
The invariant is unchanged by cyclic rotation of the ordering, by wiggles, and by proper mutations, by the invariance of the Alexander polynomial \cite[Corollary~8.5]{fomin-neville2026} read through \cite[Equation~(8.2)]{fomin-neville2026}.
For a mutation-acyclic quiver $Q$, the value of $M$ on the totally proper class of $Q$ is therefore an invariant of the mutation class of $Q$, and we denote it by $m(Q)$.
Returning to the fixed-label tuple $\qtuple{Q}=(a,b,c,d,e,f)$, the Pfaffian is
\begin{equation}\label{eq:pf}
    p(Q)=af-be+cd,\qquad \det B_Q=p(Q)^2.
\end{equation}
The factorization of matrix mutation in the proof of \cite[Lemma~3.2]{BFZ3} shows that mutation changes the sign of $p(Q)$, so $|p(Q)|$ is mutation invariant, and hence so is $\det B_Q$ \cite[Theorem~2.8.4]{FWZ}.
Relabeling gives $p(Q^\pi)=\operatorname{sgn}(\pi)p(Q)$ and hence preserves $p(Q)^2$, and reversing all arrows together with the cyclic ordering, which is passing to the opposite COQ of \cite[Definition~12.1]{fomin-neville2026}, preserves $M_\omega(Q)$ \cite[Lemma~12.4 and (8.2)]{fomin-neville2026}.

An acyclic quiver has a cyclic ordering, induced by a topological ordering, in which all its entries are nonnegative.
In that ordering the terms of the Markov invariant cannot cancel, so the Markov invariant bounds every arrow multiplicity of the acyclic quiver.
Since $m(Q)$ is a mutation invariant, it bounds the multiplicities of every acyclic quiver in the mutation class of $Q$, though not those of $Q$ itself.
The next lemma makes this precise.

\begin{lemma}\label{lem:endpoint}
    Suppose a four-vertex quiver $Q$ is mutation-acyclic, and let $m=m(Q)$.
    If $A$ is an acyclic quiver mutation equivalent to $Q$, then, for a topological ordering $v_1<\cdots<v_4$ of $A$ and $\omega_A=\corder{v_1,\dots,v_4}$, the coordinates $\oqtuple{\omega_A}{A}=(a,b,c,d,e,f)$ satisfy
    \begin{equation}\label{eq:endpoint}
        a,b,c,d,e,f\ge0,\qquad
        a^2+b^2+c^2+d^2+e^2+f^2\le m,\qquad
        2|p(A)|\le m.
    \end{equation}
    In particular, $\htq(A)\le\lfloor\sqrt m\rfloor$.
    If $Q$ is connected, then $m\ge3$ and $2|p(Q)|<m$.
\end{lemma}
The sum-of-squares bound in \eqref{eq:endpoint} is \cite[Proposition~3.21]{neville2025}, and the bound $m\ge3$ is the case $n=4$ of \cite[Corollary~6.6]{neville2025}.
The Pfaffian bounds are new.
For four vertices, \cite[Corollary~6.7]{neville2025} gives $p(Q)^2=\det B_Q\le(2m)^2$, which is $|p(Q)|\le 2m$, so \eqref{eq:endpoint} improves it by a factor of four.
We include the full proof because the strict inequality for connected quivers is used often in the appendices.
\begin{proof}
    Transport the totally proper class along a mutation sequence from $Q$ to $A$.
    By Theorem~\ref{thm:FN}, it becomes the class of a topological ordering of $A$ with the same Markov invariant $m$, and all six upper-triangular entries in that ordering are nonnegative.
    Every term on the right of \eqref{eq:markov-invariant} is then nonnegative, which gives the sum-of-squares bound, and since each entry squared is at most this sum, $\htq(A)\le\lfloor\sqrt m\rfloor$.
    Moreover,
    \[
        2|af-be+cd|\le 2(af+be+cd)\le a^2+f^2+b^2+e^2+c^2+d^2\le m,
    \]
    where the second inequality follows from repeatedly applying $2xy \leq x^2 + y^2,$ which is true because $(x-y)^2\geq 0$.

    Now suppose $Q$ is connected.
    Mutation preserves connectedness, so $A$ is connected, and in a connected quiver on four vertices at least three of the six pairs of vertices are joined.
    Hence at least three entries are positive integers and $m\ge3$.
    Since $|p(\cdot)|$ is invariant under mutation, $|p(Q)|=|p(A)|$, so it remains to show $2|p(A)|<m$.

    Suppose, for contradiction, that $2|p(A)|=m$.
    Then every inequality above becomes an equality.
    Equality in the second step forces $(a-f)^2 + (b-e)^2 + (c-d)^2=0$, so $a=f$, $b=e$, $c=d$.
    Equality in the triangle inequality gives $be=0$ or $af=cd=0$.
    In the last step, this forces every cubic and quartic term of \eqref{eq:markov-invariant} to vanish; in particular the quartic term containing $acdf$ vanishes, so $a$, $f$, $c$, $d$ cannot all be positive.
    This implies that at most one of the pairs $(a,f)$, $(b,e)$, $(c,d)$ is nonzero.
    The underlying graph of $A$ then has at most two edges, so $A$ is disconnected, a contradiction.
\end{proof}

\begin{remark}\label{rem:pfsharp}
    For each integer $k\ge1$, let $L_k$ be defined by $\qtuple{L_k}=(k,0,0,1,0,k)$.
    This is the connected acyclic quiver $1\xrightarrow{k}2\to3\xrightarrow{k}4$, with topological ordering $1<2<3<4$ and induced cyclic ordering $\omega=\corder{1,2,3,4}$.
    Direct substitution gives \claim{S2.Lk-markov, S2.Lk-pf}
    \[
        M_\omega(L_k)=2k^2+1,\qquad p(L_k)=k^2.
    \]
    Consequently, $M_\omega(L_k)-2|p(L_k)|=1$, the smallest possible positive integer gap in the strict inequality of Lemma~\ref{lem:endpoint}.
    Moreover,
    \[
        \frac{M_\omega(L_k)}{|p(L_k)|}=2+\frac{1}{k^2}\longrightarrow2.
    \]
    Thus, for every constant $c>2$, sufficiently large $k$ satisfies
    $M_\omega(L_k)<c|p(L_k)|$. The coefficient $2$ therefore cannot be increased uniformly, even when restricted to connected acyclic four-vertex quivers.
\end{remark}

Lemma~\ref{lem:endpoint} gives a finite obstruction test, which we name for later use.
We say that a connected four-vertex quiver $Q$ is \textbf{Markov-excluded} if it has at least one proper wiggle class, and for every proper wiggle class $[\omega]_{\mathrm{wig}}$ no connected acyclic four-vertex quiver $A$, with cyclic ordering $\omega_A$ induced by a topological ordering, has both $M_{\omega_A}(A)=M_\omega(Q)$ and $p(A)^2=p(Q)^2$.

\begin{lemma}\label{lem:exclusion}
    Let $Q$ be a connected four-vertex quiver.
    If $Q$ has no proper wiggle class, or if $Q$ is Markov-excluded, then $Q$ is mutation-cyclic.
    Whether $Q$ is Markov-excluded can be decided by a finite computation.
    For four-vertex quivers, having no proper wiggle class is equivalent to being a vortex, by \cite[Proposition~13.2]{fomin-neville2026}.
    The resulting obstruction to mutation-acyclicity is stated separately in Lemma~\ref{lem:vortex} below.
\end{lemma}
\begin{proof}
    If $Q$ were mutation-acyclic, Theorem~\ref{thm:FN} would give a totally proper class $[\omega]_{\mathrm{wig}}$, and Lemma~\ref{lem:endpoint} would give an acyclic quiver $A$ in the class of $Q$ with $M_{\omega_A}(A)=M_\omega(Q)$, where $\omega_A$ is induced by a topological ordering.
    Also, $p(A)^2=p(Q)^2$ by \eqref{eq:pf}, and $A$ is connected because mutation preserves components.
    Both hypotheses rule this possibility out.
    For the finiteness, put $m=M_\omega(Q)$ for a proper class $[\omega]_{\mathrm{wig}}$.
    By \eqref{eq:endpoint} every candidate tuple $\oqtuple{\omega_A}{A}$ is nonnegative with sum of squares at most $m$, so there are finitely many to check, and none if $m<3$.
\end{proof}

A match in the exclusion test does not prove mutation-acyclicity.
Neville uses the same exclusion idea in \cite[Section~6]{neville2025}.
A \textbf{vortex} is a complete four-vertex quiver consisting of an oriented triangle and a fourth vertex that is a source or a sink \cite[Definition~13.1]{fomin-neville2026}.
Fomin and Neville show that a four-vertex quiver has a proper cyclic ordering if and only if it is not a vortex \cite[Proposition~13.2]{fomin-neville2026}.
A vortex is complete, so each of its wiggle classes consists of a single cyclic ordering, and having no proper cyclic ordering is the same as having no proper wiggle class.
A quiver that is not a vortex has a proper cyclic ordering, whose wiggle class is proper.
This proves the equivalence stated in Lemma~\ref{lem:exclusion}.

\begin{lemma}[{\cite[Proposition~13.2]{fomin-neville2026}}]\label{lem:vortex}
    A vortex has no proper wiggle class.
    Hence, by Theorem~\ref{thm:FN}, every vortex is mutation-cyclic.
\end{lemma}

The following fact holds in every rank and is used twice below.

\begin{lemma}[{{\cite[Corollary~4 and its proof]{caldero2006}}}]\label{lem:acyclicreps}
    Let $A$ and $B$ be mutation-equivalent acyclic quivers of the same rank.
    Then $B$ can be reached from $A$, up to relabeling of vertices, by a sequence of mutations at sources and sinks.
    Consequently $A$ and $B$ have the same underlying weighted graph up to relabeling, and a mutation class contains only finitely many acyclic quivers.
\end{lemma}
Caldero and Keller show that the seeds with acyclic exchange matrix form a connected subgraph of the exchange graph \cite[Corollary~4]{caldero2006}, and the statement about sources and sinks is in its proof; for acyclic seeds see also \cite{BMRT2007}.
A mutation between two acyclic quivers is necessarily at a source or a sink.
Indeed, if a vertex $k$ of an acyclic quiver is neither a source nor a sink, there is a path $i\to k\to j$.
There can be no arrow $j\to i$, since that would already give a directed cycle.
Mutation at $k$ therefore leaves a positive number of arrows $i\to j$, together with the reversed arrows $j\to k\to i$, producing a directed triangle.

\section{An alternating family}\label{sec:family}

For $t\in\Z$, let
\begin{equation}\label{eq:family}
    \qtuple{Q(t)}=(-2,t+2,t,-t,2-t,0) = (-2,2,0,0,2,0)+t(0,1,1,-1,-1,0),
\end{equation}
where negative labels in the following diagram mean reversed arrows and zero labels mean omitted arrows:
\[Q(t) = \begin{tikzcd}
    & 1 & \\
    3 & & 4 \\
    & 2 &
    \arrow[from=3-2, to=1-2, "2" description]
    \arrow[from=1-2, to=2-1, "t+2"']
    \arrow[from=1-2, to=2-3, "t"]
    \arrow[from=2-1, to=3-2, "t"']
    \arrow[from=2-3, to=3-2, "t-2"]
    \end{tikzcd}\]
Each $Q(t)$ is connected, since the edge between vertex $1$ and $2$ is always present, vertex $3$ is joined to $1$ unless $t=-2$ and to $2$ unless $t=0$, and vertex $4$ is joined to $1$ unless $t=0$ and to $2$ unless $t=2$.

\begin{lemma}\label{lem:recurrence}
    For $t\ge2$, mutating $Q(t)$ at vertex $2$ gives $Q(t-2)^{\pi_{12}}$.
\end{lemma}
\begin{proof}
    Mutation at $2$ gives \claim{S3.recurrence}
    \[
        \qtuple{\mu_2Q(t)}=(2,2-t,4-t,t,t-2,0).
    \]
    Swapping the first two labels gives
    \[
        \qtuple{(\mu_2Q(t))^{\pi_{12}}}=(-2,t,t-2,2-t,4-t,0)=\qtuple{Q(t-2)}.\qedhere
    \]
\end{proof}

\begin{theorem}\label{thm:parity}
    The four-vertex quiver $Q(t)$ is mutation-acyclic if and only if $t<0$ or $t$ is even.
\end{theorem}
\begin{proof}
    For $t<0$, vertex $2$ is a source and there is no arrow between vertices $3$ and $4$.
    Hence $Q(t)$ is acyclic; for $t\le-3$ this is \claim{S3.Qneg}, and \file{examples.py} checks $t=-2,-1$.

    For $t\ge0$, repeatedly apply Lemma~\ref{lem:recurrence}.
    Each step subtracts two from $t$ after relabeling, which preserves mutation-acyclicity.
    So we only need to understand $Q(0)$ and $Q(1)$.

    The even endpoint $Q(0)$, with $\qtuple{Q(0)}=(-2,2,0,0,2,0)$, is acyclic.
    For the odd endpoint, one mutation gives
    \[
        \qtuple{\mu_2Q(1)}=(2,1,1,1,-1,1).
    \]
    Here vertex $1$ is a source, while $2\to3\to4\to2$ is a directed triangle.
    Every pair of vertices is joined, so this is a vortex.
    Lemma~\ref{lem:vortex} proves that $Q(1)$ is therefore mutation-cyclic.
\end{proof}

\begin{remark}\label{rem:familymarkov}
    For every $t\ge0$, the quiver $Q(t)$ has exactly one proper wiggle class, and its Markov value is $12$.
    For $t\ge3$ this class contains $\corder{1,3,4,2}$ \claim{S3.Qt-class, S3.Qt-markov}, and for $t\le2$ it is a direct check.
    Also $p(Q(t))=-4$ for every $t$ \claim{S3.Qt-pf}.
    The connected acyclic quiver $Q(0)$ has the same pair $(M,p^2)=(12,16)$ in its topological ordering.
    Hence no $Q(t)$ is Markov-excluded, and the even and odd members of the family are not distinguished by the Markov invariant, by the determinant, or by properness.
    The odd members become detectable after mutation.
    By Lemma~\ref{lem:recurrence}, $s$ mutations take $Q(2s+1)$ to a relabeling of $Q(1)$, and one more mutation gives a vortex.
\end{remark}

\begin{remark}\label{rem:rank}
    For $n>4$, build $Q_n(t)$ from $Q(t)$ by adding the vertices $5,\dots,n$ and the arrows
    \[
        n\longrightarrow n-1\longrightarrow\cdots\longrightarrow5\longrightarrow3,
    \]
    with the convention that for $n=5$ the only added arrow is $5\to3$.
    In particular, no arrow $4\to3$ is added, and the full subquiver on $\{1,2,3,4\}$ is $Q(t)$.
    The resulting quiver is connected.
    None of the added vertices is adjacent to vertex $2$, so mutation at $2$ leaves all the added arrows unchanged, as does the swap of labels $1,2$.
    Consequently, the recurrence of Lemma~\ref{lem:recurrence} sends $Q_n(t)$ to $Q_n(t-2)$ for $t\ge2$.

    For $t<0$, the original quiver is acyclic, and attaching a chain at a single vertex creates no directed cycle.
    The same argument applies to the even endpoint $Q_n(0)$.
    For the odd endpoint, the full subquiver on $\{1,2,3,4\}$ is the mutation-cyclic quiver $Q(1)$, so Proposition~\ref{prop:rank3} rules out mutation-acyclicity of $Q_n(1)$.
    It follows that $Q_n(t)$ is mutation-acyclic exactly when $t<0$ or $t$ is even.
    We also write $Q_4(t)=Q(t)$.
\end{remark}

\subsection{Polynomial sign tests}\label{sec:signs}

We now prove Theorem~\ref{thm:nogo}.
The point is that finitely many polynomial tests eventually see no change along the family, whereas mutation-acyclicity continues to alternate.

\begin{proof}[Proof of Theorem~\ref{thm:nogo}]
    Fix $n\ge4$, and suppose for contradiction that such a rule exists for quivers on $n$ vertices.
    Move all terms of each polynomial equation/inequality to one side, so that it has the form $P=0$, $P>0$, or $P<0$ for some polynomial $P$.
    Because the rule is fixed and finite, there are finitely many polynomials
    \[
        P_1,\dots,P_r\in\mathbb{R}[b_{ij}:1\le i<j\le n]
    \]
    whose zero values and signs determine its answer.
    Combining these tests with the Boolean operations \emph{and}, \emph{or}, and \emph{not} uses no information beyond this finite list of signs.
    If there are no tests, the rule is constant and already cannot distinguish the even and odd members of the family.

    Now apply the rule only to the connected integral quivers $Q_n(t)$ of Remark~\ref{rem:rank}, for $t=0,1,2,\dots$.
    For each $\ell$, the value tested by the rule is
    \[
        q_\ell(t):=P_\ell\bigl(B_{Q_n(t)}\bigr).
    \]
    Every entry of $B_{Q_n(t)}$ is an affine function of $t$, since the entries of $Q(t)$ are affine in $t$ by \eqref{eq:family} and the added entries are constant.
    So $q_\ell$ is an ordinary one-variable real polynomial.
    There are exactly two possibilities:
    \begin{enumerate}[label=(\roman*)]
        \item If $q_\ell$ is the zero polynomial, then $q_\ell(t)=0$ for every $t$.
        Its sign is always zero, so all tests involving $P_\ell$ have constant answers on the family.
        \item If $q_\ell$ is not the zero polynomial, it has only finitely many real roots.
        Choose a real number $R_\ell$ larger than all of them, taking $R_\ell=0$ if there are none.
        The polynomial never vanishes on the interval $(R_\ell,\infty)$.
        It is either positive everywhere on that interval or negative everywhere there.
        Otherwise continuity and the intermediate value theorem would give another root between two points of opposite signs.
        Thus its sign, and every test involving $P_\ell$, is constant for $t>R_\ell$.
    \end{enumerate}
    Since there are only finitely many $\ell$, we can choose a single integer $N\geq2$ larger than every $R_\ell$ needed in case~(ii).
    For all integers $t\geq N$, the entire sign vector
    \[
        \bigl(\operatorname{sgn}q_1(t),\dots,\operatorname{sgn}q_r(t)\bigr)
    \]
    is constant.
    In particular, every individual test has the same truth value on every such $Q_n(t)$ with $t\geq N$.
    Applying the same fixed Boolean combination to the same truth values must give the same final answer.

    Choose an even integer $2s\ge N$.
    The proposed rule therefore gives the same answer on $Q_n(2s)$ and $Q_n(2s+1)$.
    But Theorem~\ref{thm:parity} and Remark~\ref{rem:rank} say that $Q_n(2s)$ is mutation-acyclic and $Q_n(2s+1)$ is mutation-cyclic.
    The rule must be wrong on at least one of these two connected integral quivers.
    This contradicts its assumed correctness on every input, proving the theorem.
\end{proof}

\begin{remark}\label{rem:signscope}
    Equivalently, for each $n\ge4$, no semialgebraic subset of $\mathbb{R}^{\binom n2}$ agrees with mutation-acyclicity on all integral $n$-vertex quivers.
    This assertion concerns the integral points; no notion of mutation-acyclicity for real weights is needed.
    It also rules out fixed finite piecewise-polynomial rules with semialgebraic branches, and rational sign tests after denominator signs and zero denominators are accounted for.

    Arithmetic tests are a different story.
    For example, the greatest common divisor of the entries of the exchange matrix is a mutation invariant \cite[Section~2.8]{FWZ}, and more generally, for skew-symmetric matrices mutation is an integral congruence, so the whole Smith normal form is invariant \cite[Section~1]{casals2024}.
    For the family $Q(t)$ the greatest common divisor equals $\gcd(2,t)$, which distinguishes positive even and positive odd parameters.
    Finer arithmetic invariants include the binary invariant of Casals \cite{casals2024} and the congruence invariants of Seven and \"Unal \cite{seven-unal2025}.
    The theorem does not exclude such tests, or the finite decision algorithm proved below.
\end{remark}

\begin{corollary}\label{cor:nobound}
    There is no uniform bound on the length of a shortest mutation sequence from a mutation-acyclic four-vertex quiver to an acyclic quiver.
\end{corollary}
\begin{proof}
    By Lemma~\ref{lem:recurrence}, the distinct quivers $Q(2s)$, $s\ge0$, lie in the mutation class of $Q(0)$ up to relabeling.
    By Lemma~\ref{lem:acyclicreps}, that class contains only finitely many acyclic quivers, and there are finitely many relabelings.
    From finitely many endpoints, at most $L$ mutations can reach only finitely many quivers, because at each step there are only four mutation vertices to choose from.
    Reversing these paths shows that a fixed length bound $L$ cannot cover the infinitely many distinct quivers $Q(2s)$.
\end{proof}

\begin{remark}\label{rem:nogorank}
    Corollary~\ref{cor:nobound} holds for connected integral quivers of every rank $n\ge4$, with the same proof applied to the family $Q_n(t)$ of Remark~\ref{rem:rank}, since Lemma~\ref{lem:acyclicreps} supplies the finiteness of acyclic representatives in every rank.
    The results that follow, by contrast, use the four-vertex Markov invariant \eqref{eq:markov-invariant} in an essential way, and therefore apply to rank four only.
\end{remark}

\section{Local replacements along mutation sequences}\label{sec:local}

This section sets up the technical input for the height bound of Theorem~\ref{thm:height}.
The guiding idea is that whenever a mutation sequence climbs above a certain height level and later comes back down, the stretch spent above that level (a ``\textbf{high segment},'' defined below) can be rerouted below it.
The three lemmas here justify this locally by following how the first two mutation vertices of the high segment are joined.
With no edge between them the two mutations commute; with one arrow the pentagon relation supplies detours; with at least two arrows an attempted detour is forced to keep growing.
These three regimes mirror the behavior of ice quivers with exactly two mutable vertices.
Soukup shows that the mutation class of such an ice quiver has at most four quivers when the two mutable vertices are not joined and at most ten when they are joined by a single arrow, while in any nontrivial case the arrow multiplicities grow without bound when they are joined by two or more arrows \cite[Theorem~1.3]{soukup2025}.
Here the additional task is to control the height along a mutation sequence (what we will call a \textbf{route}) and the possibility of mutating at a third vertex.
The proofs are case analyses, given in Appendices~\ref{app:lemmas} and~\ref{app:signs}, and every calculation in them is verified by the program \file{check\_proofs.py}.

\subsection{The setup and supporting lemmas}

Throughout this section and Appendices~\ref{app:lemmas} and~\ref{app:signs}, fix a connected mutation-acyclic four-vertex quiver $Q_*$, let $m=m(Q_*)$, fix an integer $T\ge m$, and put $F=T+1$.
All quivers considered below are mutation-equivalent to $Q_*$.
Call a quiver \textbf{low} if its height is at most $T$ and \textbf{high} if its height is exactly $F$.
The routes we choose stay within height $F$, so each quiver on them is low or high; quivers of larger height appear only as intermediate quivers of alternative routes that we compare.
A \textbf{high segment} is a maximal run of consecutive high quivers on a route.
If it is preceded and followed by low quivers, its associated \textbf{excursion word} starts with the mutation from the low predecessor to the first high quiver and ends with the mutation from the last high quiver to the low successor.
Thus a segment containing $r$ high quivers has an excursion word of length $r+1$.
When we describe the mutation vertices or replace the word of a high segment, we always mean this excursion word.

\begin{lemma}\label{lem:prefix}
    Let $R$ be mutation-equivalent to $Q_*$, and let $p=p(R)$.
    Then
    \begin{equation}\label{eq:prefixscalar}
        3\le m\le T,\qquad 2|p|<m.
    \end{equation}
    In particular $T\ge3$ and $F\ge4$.
    These bounds are unchanged by relabeling $R$ and by reversing all arrows of $R$ (together with its cyclic ordering).
\end{lemma}
\begin{proof}
    By Theorem~\ref{thm:FN}, the totally proper class of $R$ is obtained by transporting that of $Q_*$ along a mutation sequence, so its Markov invariant is $m(R)=m(Q_*)=m$.
    Lemma~\ref{lem:endpoint} applied to the connected quiver $Q_*$ gives $m\ge3$ and $2|p(Q_*)|<m$, and $|p(R)|=|p(Q_*)|$ because $|p|$ is mutation invariant.
    The bound $m\le T$ holds by the choice of $T$.
    The invariance statements were noted after equation~\eqref{eq:pf}.
\end{proof}
These bounds always refer to the totally proper class, and a different proper cyclic ordering need not give the same value.
Much of the work in the appendices consists of identifying that class.

We also record two elementary tools.
First, for integers $1\le r,s\le F-1$,
\begin{equation}\label{eq:rectangle}
    r^2+s^2-Frs\le2-F.
\end{equation}
Indeed, the left side is convex in each variable, so its maximum on the rectangle occurs at a corner.
The corners $(1,1)$, $(1,F-1)$ and $(F-1,1)$ give $2-F$; the corner $(F-1,F-1)$ gives $(2-F)(F-1)^2\leq 2-F$ \claim{A0.rect}.

Second, if the edge between vertices $1$ and $2$ of a quiver has multiplicity one, the type-$A_2$ pentagon relation \cite{FZ1}, in the form for extended exchange matrices with arbitrary additional rows \cite[Example~3.2.4]{FWZ}, gives
\begin{equation}\label{eq:pentagon}
    \mu_{\mword{1,2,1,2,1}}Q=Q^{\pi_{12}}.
\end{equation}
It holds regardless of the signs of the other entries.
To compute the intermediate quivers, apply the mutation formula \eqref{eq:mutation} to columns $3$ and $4$; the $b_{34}$ entry then follows from $p(\mu_kQ)=-p(Q)$ and the $b_{12}$ entry, which is $1$ or $-1$.

\subsection{The three local lemmas}

In Lemmas~\ref{lem:zero} and~\ref{lem:strong}, $P$ is a high quiver mutation-equivalent to $Q_*$, and $i\ne j$ are vertices with
\[
    \htq(\mu_iP)\leq T,\qquad \htq(\mu_jP)\leq F.
\]
This is the situation at the start of a high segment, where $\mu_iP$ is the low quiver the route came from and $j$ is the next mutation vertex.

\begin{lemma}[No arrow]\label{lem:zero}
    Suppose there is no edge between $i$ and $j$ in $P$.
    Then a reduced sequence starting at $P$ with $j$, staying within height $F$, and stopping at its first low quiver is either $\mathbf w=\mword{j}$ or $\mathbf w=\mword{j,i}$.
    Moreover, if $R$ is a low quiver with no edge between $i,j$, then every low quiver obtained from $R$ using only mutations at $i$ and $j$ can be reached from $R$ by at most two such mutations with every intermediate quiver low.
\end{lemma}
Since $i,j$ are not joined, mutations at $i$ and $j$ commute, and the only entry either of these can push above $T$ is the one between the other two vertices.
The lemma implies that the commuting square never needs to pass through a high corner.
The two routes in this square have the same labeled endpoint:
\[
\begin{tikzcd}[column sep=large,row sep=large]
    R \arrow[r,"\mu_i"] \arrow[d,"\mu_j"'] &
    \mu_iR \arrow[d,"\mu_j"]\\
    \mu_jR \arrow[r,"\mu_i"'] & \mu_{\mword{i,j}}R.
\end{tikzcd}
\]

\begin{remark}\label{rem:zero-rank-five}
    The first assertion of Lemma~\ref{lem:zero} does not extend verbatim to five-vertex quivers.
    For a five-vertex quiver $Q$ write
    \[
        \qtuple{Q}=(b_{12},b_{13},b_{14},b_{15},b_{23},b_{24},b_{25},b_{34},b_{35},b_{45}),
    \]
    and let $m(Q)$ be the value on the totally proper class of the Fomin--Neville invariant $M=5-\operatorname{tr}(U^{-T}U)$ \cite[Equation~(8.2)]{fomin-neville2026}, where $U$ is the unipotent upper triangular matrix with $U_{ij}=-b_{ij}$ for $i<j$.
    On four vertices, $4-\operatorname{tr}(U^{-T}U)$ is the invariant \eqref{eq:markov-invariant}.
 
    Consider the connected quiver $C$ with $\qtuple{C}=(1,-3,0,3,0,0,2,0,-15,2)$.
    Then $\qtuple{\mu_{\mword{1,4,2}}C}=(1,0,0,-3,3,0,-2,0,0,-2)$, which is acyclic with topological ordering $5<1<2<3<4$ and Markov value $33$.
    So $C$ is mutation-acyclic and $m(C)=33$.
    Take $T=41$, $F=42$, and $P=\mu_3C$.
    The heights of $C$ and of the quivers along $\mword{3,4,2}$ are $15,42,42,40$.
    Moreover $b_{34}(P)=0$, $\htq(\mu_3P)=15\le T$ and $\htq(\mu_4P)=F$.
    The reduced continuation $\mword{4,2}$ from $P$ therefore reaches its first low quiver within height $F$, although it is neither $\mword{4}$ nor $\mword{4,3}$.
 
    The excursion still has a low replacement.
    The vertices $2,3,4$ are pairwise nonadjacent in $C$, so their mutations commute, and $\mu_{\mword{3,4,2}}C=\mu_{\mword{2,3,4}}C$.
    The heights along the replacement are $15,15,40,40$.
\end{remark}

\begin{lemma}[A single arrow]\label{lem:unitroute}
    Let $K\geq0$ be an integer, and let $C$ be a quiver mutation-equivalent to $Q_*$ with exactly one arrow between $i$ and $j$.
    Suppose that $C$ and a quiver $R$ obtained from $C$ by mutations at $i$ and $j$ both have height at most $K$.
    Then a sequence of at most three mutations at $i,j$ reaches $R$, possibly with $i,j$ swapped, and every quiver along it has height at most $K$.
\end{lemma}
Here the pentagon relation \eqref{eq:pentagon} gives a ten-quiver picture, and any endpoint can be reached by two routes.
This lemma says at least one of them stays within height $K$.
For example, if $R=\mu_{\mword{1,2}}C$, the two routes to the same endpoint up to $\pi_{12}$ are
\[
\begin{tikzcd}[column sep=large,row sep=large]
    C \arrow[r,"\mu_1"] \arrow[d,"\mu_2"'] &
    \mu_1C \arrow[r,"\mu_2"] & R\\
    \mu_2C \arrow[r,"\mu_1"'] &
    \mu_{\mword{2,1}}C \arrow[r,"\mu_2"'] & R^{\pi_{12}}.
\end{tikzcd}
\]
The direct route has length two and the detour has length three, and the lemma controls their intermediate heights.

\begin{lemma}[Multiplicity at least two]\label{lem:strong}
    Suppose the edge between $i$ and $j$ in $P$ has multiplicity at least two.
    Then no reduced sequence starting at $P$ with $j$ can reach a low quiver while staying within height $F$.
\end{lemma}
Here a double arrow makes the entries grow.
After the mutation at $j$, every mutation other than at $i$ exceeds height $F$, and alternating at $i,j$ never lowers the entry that is already $F$.

\begin{remark}\label{rem:strong-rank-five}
    Lemma~\ref{lem:strong} also fails for five-vertex quivers.
    With the conventions of Remark~\ref{rem:zero-rank-five}, consider the connected quiver $C$ with $\qtuple{C}=(1,-2,0,0,2,-1,0,0,6,-4)$.
    Then $\qtuple{\mu_{\mword{2,4,3}}C}=(-1,0,0,0,0,-1,0,-2,2,0)$, which is acyclic with topological ordering $4<2<1<3<5$ and Markov value $10$.
    So $C$ is mutation-acyclic and $m(C)=10$.
    Take $T=11$, $F=12$, and $P=\mu_3C$.
    Then
    \[
    \begin{aligned}
        \qtuple{P}&=(-3,2,0,0,-2,-1,12,0,-6,-4),\\
        \qtuple{\mu_1P}&=(3,-2,0,0,4,-1,12,0,-6,-4),\\
        \qtuple{\mu_{\mword{1,4}}P}&=(3,-2,0,0,4,1,8,0,-6,4).
    \end{aligned}
    \]
    Thus $|b_{13}(P)|=2$, $\htq(\mu_3P)=\htq(C)=6\le T$ and $\htq(\mu_1P)=F$, but the reduced continuation $\mword{1,4}$ reaches a quiver of height $8\le T$ within height $F$.
 
    Again the excursion has a low replacement.
    Vertex $4$ is joined to neither $1$ nor $3$ in $C$, so its mutation commutes with both, and $\mu_{\mword{3,1,4}}C=\mu_{\mword{4,3,1}}C$.
    The heights along the two words are $6,12,12,8$ and $6,6,8,8$.
    So the arguments behind Lemmas~\ref{lem:zero} and~\ref{lem:strong} would need further normalization in rank five, although in both examples the excursion can still be lowered.
\end{remark}

\subsection{High segments across a single arrow}

When the first two mutation vertices of a high segment are joined by a single arrow, Lemma~\ref{lem:unitroute} alone does not suffice, because the segment may use a third vertex.
The following proposition classifies the remaining cases for a minimizing route, defined next.
The restriction to minimizing routes is essential, as Remark~\ref{rem:unitminimality} shows.

For a finite route $\mathcal R=(R_0,\dots,R_L)$, let
\[
    F(\mathcal R)=\max_{0\leq j\leq L}\htq(R_j),\quad
    N(\mathcal R)=\#\{j:\htq(R_j)=F(\mathcal R)\}.
\]
We call $\mathcal R$ \textbf{minimizing} if, among all mutation routes from its fixed starting quiver to any acyclic quiver, it minimizes lexicographically
\begin{equation}\label{eq:routecost}
    \bigl(F(\mathcal R),N(\mathcal R),L\bigr).
\end{equation}
In other words, first minimize the maximum height, then the number of visits at that height, and finally the length.
Such a route exists whenever the starting quiver is mutation-acyclic, since each quantity is a nonnegative integer.
The last minimum makes the route reduced.
Relabeling the entire route or reversing every arrow preserves these three quantities and gives a bijection with routes from the transformed starting quiver, so it also preserves minimality.

The quivers defined by
\begin{equation}\label{eq:AI}
    \qtuple{A_c}=(1,-2,c,1,0,F-2c),\quad
    \qtuple{I_{t,u}}=(1,-1,-t,2,t-F,F-u),\quad 1\leq c,t,u\leq T,
\end{equation}
are the only possible low predecessors when a high segment of a minimizing route uses a third mutation vertex, and each admits an explicit low detour.

\begin{proposition}\label{prop:unit}
    Let $\mathcal R$ be a minimizing mutation route from $Q_*$ to an acyclic quiver in the sense of \eqref{eq:routecost}, and suppose its maximum height is $F$.
    Consider a high segment of $\mathcal R$, preceded by a low quiver $C$ and followed by a low quiver.
    Suppose exactly one arrow joins the first two mutation vertices of the segment.
    Relabel so that these vertices are $1$ and $2$, in that order, and reverse all arrows if necessary so that $b_{12}(C)=1$.
    Then one of the following holds.
    \begin{enumerate}[label=(\alph*)]
        \item The segment mutates only at $1$ and $2$.
        Then a sequence of at most three mutations at $1,2$, all of whose quivers are low, reaches the same low endpoint, possibly relabeled by the swap of $1,2$.
        \item The segment mutates at a third vertex, which we relabel $3$.
        Then the segment begins with exactly $\mathbf w=\mword{1,2}$ before its first mutation at $3$, the quiver $C$ is $A_c$ for some $1\leq c\leq T$ or $I_{t,u}$ for some $1\leq t\ne u\leq T$, and the segment is replaced as follows:
        \begin{center}
        \begin{tabular}{c|c|c|c}
             starting quiver & segment & replacement & endpoint of replacement\\\hline
             $A_c$ & $\mword{1,2,3,1,3}$ & $\mword{2,1,3}$ & $(\mu_{\mword{1,2,3,1,3}}A_c)^{\pi_{123}}$\\
             $I_{t,u},\ t>u$ & $\mword{1,2,3}$ & $\mword{2,1,3,2,3}$ & $(\mu_{\mword{1,2,3}}I_{t,u})^{\pi_{123}}$\\
             $I_{t,u},\ t < u$ & $\mword{1,2,3}$ & $\mword{3,1,3,2,1}$ & $(\mu_{\mword{1,2,3}}I_{t,u})^{\pi_{123}}$
        \end{tabular}
        \end{center}
        In each row the segment column is the excursion word from $C$ through the entire high segment to its first low successor, and every quiver on the replacement is low.
    \end{enumerate}
\end{proposition}
Part (a) is Lemma~\ref{lem:unitroute} with $K=T$.
The rest is proved in Appendix~\ref{app:signs}. Appendix~\ref{app:prefix} shows that the alternating prefix at $1,2$ is exactly $\mathbf w=\mword{1,2}$, Appendices~\ref{app:case1}--\ref{app:case3} run through the signs of the entries of $C$ to isolate $A_c$ and $I_{t,u}$, and Appendix~\ref{app:equal} verifies the replacements and excludes $t=u$.

\begin{remark}\label{rem:unitminimality}
The minimizing-route hypothesis in Proposition~\ref{prop:unit} cannot be replaced by reducedness alone.
Take the connected acyclic quiver
\[
    \qtuple{Q_*}=(1,1,-1,1,0,0),\qquad m(Q_*)=T=5,\qquad F=6.
\]
The word $\mathbf w=\mword{2,3,1,3,1,3,1,3}$ reaches the quiver $C$ with $\qtuple{C}=(1,-2,5,1,-1,-4)$, with successive heights $2,2,2,2,2,3,4,5$.
Starting at this low quiver $C$, the following reduced excursion remains within height six. In the row for $\mathbf v$, each entry column gives $b_{ij}(\mu_{\mathbf v}C)$:
\begin{center}\small
\begin{tabular}{l|rrrrrr|r}\toprule
    Word from $C$ & $b_{12}$ & $b_{13}$ & $b_{14}$ & $b_{23}$ & $b_{24}$ & $b_{34}$ & Height\\\midrule
    $\mword{}$ & $1$ & $-2$ & $5$ & $1$ & $-1$ & $-4$ & $5$\\
    $\mword{1}$ & $-1$ & $2$ & $-5$ & $-1$ & $-1$ & $6$ & $6$\\
    $\mword{1,2}$ & $1$ & $1$ & $-6$ & $1$ & $1$ & $6$ & $6$\\
    $\mword{1,2,1}$ & $-1$ & $-1$ & $6$ & $1$ & $-5$ & $0$ & $6$\\
    $\mword{1,2,1,3}$ & $-2$ & $1$ & $6$ & $-1$ & $-5$ & $0$ & $6$\\
    $\mword{1,2,1,3,2}$ & $2$ & $-1$ & $-4$ & $1$ & $5$ & $0$ & $5$\\\bottomrule
\end{tabular}
\end{center}
After the return to height five, the word $\mword{1,2,1,2,3}$ reaches the acyclic quiver $(1,1,0,-1,1,0)$, with successive heights $4,3,2,2,1$.
Thus the full word
\[
    \mathbf w=\mword{2,3,1,3,1,3,1,3}\,\mword{1,2,1,3,2}\,\mword{1,2,1,2,3}
\]
is reduced, starts at $Q_*$, ends at an acyclic quiver, and has maximum height six.
The first two excursion vertices are $1,2$, with $b_{12}(C)=1$, but the first mutation at a third vertex occurs after the prefix $\mword{1,2,1}$ rather than $\mword{1,2}$.
Moreover, $C$ is neither $A_c$ nor $I_{t,u}$ in this labeling.
This contradicts the conclusion of Proposition~\ref{prop:unit} if its minimality hypothesis is omitted.
The route here is not minimizing since its starting quiver is already acyclic.
In Appendix~\ref{app:prefix}, minimality is exactly what allows longer alternating prefixes to be replaced before the sign cases are classified.
These claims are checked by \file{examples.py}.
\end{remark}

\section{A height bound and decision algorithms}\label{sec:main}

\subsection{A height bound on mutation sequences}
\begin{theorem}\label{thm:height}
    Let $Q$ be a mutation-acyclic four-vertex quiver.
    Then some mutation sequence takes $Q$ to an acyclic quiver with every intermediate arrow multiplicity at most $\max\{\htq(Q),m(Q)\}$.
\end{theorem}
\begin{proof}
    If $Q$ is disconnected, every component has at most three vertices, and components with at most two vertices are acyclic.
    A non-acyclic three-vertex component is an oriented triangle, and the descent of \cite[Section~2.5]{assem2006} mutates it at the vertex opposite a largest multiplicity until it becomes acyclic.
    If the multiplicities are $x\le y\le z$, such a mutation keeps $x,y$ and replaces $z$ by $|xy-z|$, and the result is acyclic if $xy\le z$.
    We claim that $xy<2z$ whenever $x\ge2$.
    Otherwise $y\le z\le xy/2$, and since $z^2-xyz$ decreases for $z\le xy/2$, the Markov constant satisfies
    \[
        x^2+y^2+z^2-xyz\le x^2+2y^2-xy^2=x^2+(2-x)y^2\le4,
    \]
    using $y\ge x\ge2$; by Proposition~\ref{prop:rank3} the component would be mutation-cyclic.
    If $x\le1$, then $xy\le z$.
    So the replaced multiplicity never increases, and while the component remains an oriented triangle the sum of its multiplicities strictly decreases, so the descent terminates.
    Apply these componentwise sequences independently, obtaining a route of height at most $\htq(Q)$.
    We can therefore assume $Q$ is connected.

    Choose a minimizing route $\mathcal R$ in the sense of \eqref{eq:routecost}, and let $F$ be its maximum height.
    Suppose, for a contradiction, that $F>\max\{\htq(Q),m(Q)\}$, and put $T=F-1$.
    Then $T\ge m(Q)$, so we are in the setup of Section~\ref{sec:local} with $Q_*=Q$.
    We will replace one high excursion by a low route.
    If this lowers the maximum height, it contradicts the first minimum in \eqref{eq:routecost}, and if the maximum remains $F$, the replacement decreases the number of visits at height $F$ and contradicts the second minimum.

    The starting quiver is low, since $\htq(Q)\leq T$.
    The ending quiver is also low, since Lemma~\ref{lem:endpoint} bounds its height by $\lfloor\sqrt{m(Q)}\rfloor\le m(Q)\le T$.
    So the route contains a high segment, preceded by a low quiver $C$ and followed by a low quiver.
    Let $i$ be the mutation producing the first high quiver $P=\mu_iC$, and $j\ne i$ the next mutation vertex. Then $\htq(\mu_iP)=\htq(C)\leq T$ and $\htq(\mu_jP)\leq F$, so the lemmas of Section~\ref{sec:local} apply to $P$, $i$, $j$.
    We distinguish by the edge between $i$ and $j$ in $P$ (its multiplicity is the same in $C$).

    \emph{Multiplicity at least two.} Lemma~\ref{lem:strong} says the route cannot return to a low quiver within height $F$.
    But it does return. This case cannot occur.

    \emph{No edge.} By the first part of Lemma~\ref{lem:zero}, the whole segment together with its return is $i\,j$ or $i\,j\,i$, so the low endpoint is reached from $C$ using only $i,j$.
    By the second part, the same labeled endpoint is reached from $C$ by at most two mutations with every quiver low.
    Replacing the segment removes its high quivers without changing anything later.

    \emph{Multiplicity one.} Relabel $i,j$ as $1,2$ and reverse all arrows if necessary, so that $b_{12}(C)=1$; by Lemma~\ref{lem:prefix} and the remarks before it, this changes neither the heights along the route nor the bounds \eqref{eq:prefixscalar}.
    The route remains minimizing after these transformations, so Proposition~\ref{prop:unit}, including its minimality hypothesis, gives a replacement for the segment in every case.
    In case (a) it is a low route of at most three mutations, and in case (b) it is one of the three listed detours, all of whose quivers are low.
    Each replacement ends at the old endpoint up to a stated relabeling $\pi$.
    Relabel every later mutation vertex $k$ as $\pi^{-1}(k)$; by \eqref{eq:tailtransport} every later quiver is relabeled by $\pi$, so all later heights and the acyclicity of the final quiver are preserved.

    In the second and third cases the modified route either has maximum height less than $F$, or still has maximum $F$ and has strictly fewer visits at that height.
    This contradicts the first or second minimum in \eqref{eq:routecost}, respectively.
    Hence $F\leq \max\{\htq(Q),m(Q)\}$.
\end{proof}

The proof uses the bound $T\ge m(Q)$ only through Lemma~\ref{lem:prefix} and to make the last quiver of the route low, and the bound $T\ge\htq(Q)$ only to make the first quiver low.

\subsection{A finite decision algorithm}
For a four-vertex quiver $Q$ with at least one proper wiggle class, put
\begin{equation}\label{eq:searchbound}
    H(Q)=\max\Bigl(\{\htq(Q)\}\cup\bigl\{M_\omega(Q):[\omega]_{\mathrm{wig}}\text{ a proper wiggle class of }Q\bigr\}\Bigr).
\end{equation}
There are six cyclic orderings to examine, written as $\corder{1,v_2,v_3,v_4}$, where $(v_2,v_3,v_4)$ ranges over the permutations of $(2,3,4)$.

\begin{corollary}\label{thm:algorithm}
    Let $Q$ be a four-vertex quiver.
    If $Q$ has no proper wiggle class, it is mutation-cyclic.
    Otherwise, $Q$ is mutation-acyclic if and only if a sequence of mutations through quivers of height at most $H(Q)$ takes $Q$ to an acyclic quiver, and such a sequence can be chosen of length at most $(2H(Q)+1)^6-1$.
    In particular, mutation-acyclicity of four-vertex quivers is decidable by a finite search.
\end{corollary}
\begin{proof}
    If $Q$ has no proper wiggle class, it is mutation-cyclic by Theorem~\ref{thm:FN}, since a totally proper class is proper.
    If $Q$ is mutation-acyclic, its totally proper class is one of its proper classes, so $m(Q)\le H(Q)$, and Theorem~\ref{thm:height} gives a sequence to an acyclic quiver within height $H(Q)$.
    There are $(2H(Q)+1)^6$ labeled integer quivers of height at most $H(Q)$.
    Starting at $Q$, search its connected component in the mutation graph restricted to these quivers.
    Record each quiver once and test it for acyclicity.
    If an acyclic quiver is found, the recorded mutation sequence proves mutation-acyclicity.
    If the component is exhausted, Theorem~\ref{thm:height} proves that $Q$ is mutation-cyclic.
    The search examines at most four mutations at each recorded quiver.
    A shortest path cannot repeat a quiver, giving the length bound.
\end{proof}

The search can be implemented as follows, storing a predecessor and a mutation vertex for each newly found quiver.
\begin{enumerate}[label=\arabic*.]
    \item If $Q$ has no proper wiggle class, return \emph{mutation-cyclic}.
    Otherwise compute $H=H(Q)$, initialize a queue with $Q$, and mark $Q$ as visited.
    \item If the queue is empty, return \emph{mutation-cyclic}.
    Otherwise remove its next quiver $R$.
    \item If $R$ is acyclic, return \emph{mutation-acyclic} together with the path reconstructed from the predecessor records.
    \item For each vertex $v\in\{1,2,3,4\}$, compute $S=\mu_vR$.
    If $\htq(S)\leq H$ and $S$ has not been visited, mark it, store its predecessor $R$ and mutation vertex $v$, and add it to the queue.
    Return to step~2.
\end{enumerate}
The negative answer in step~2 is returned only after the entire reachable component within the fixed height bound has been exhausted.
Theorem~\ref{thm:height} is what makes a negative answer conclusive.
In practice, one can first test acyclicity, mutation-cyclic full triangles, vortices, and the exclusions of Lemma~\ref{lem:exclusion}, and can process disconnected components separately.
Caching acyclic endpoint signatures $(M,p^2)$ and quotienting states by relabeling can further reduce repeated work; to recover a labeled mutation word, book-keep the relabeling used at each step.
The function \file{decide} in \file{quiver4.py} implements this procedure.

\subsection{The size of the search space}
\begin{corollary}\label{cor:complexity}
    If a four-vertex quiver $Q$ with a proper wiggle class has height $h$, then
    \[
        H(Q)\leq h^4+4h^3+6h^2.
    \]
    For $h\geq1$, the decision search of Corollary~\ref{thm:algorithm} has $O(h^{24})$ states and can be implemented in a time polynomial in $h$.
    This is a pseudo-polynomial bound, since it is polynomial in the numerical arrow multiplicities, but may be exponential in the number of bits used to write them.
    When $h=0$, the quiver is already acyclic.
\end{corollary}
\begin{proof}
    In every cyclic ordering, the six entries have absolute value at most $h$.
    Formula~\eqref{eq:markov-invariant} has six quadratic terms, four cubic terms and one quartic term.
    The triangle inequality therefore gives
    \[
        |M_\omega(Q)|\leq 6h^2+4h^3+h^4.
    \]
    For integer $h\geq 1$, the right side also bounds $h$, so it bounds $H(Q)$.
    Corollary~\ref{thm:algorithm} now gives at most
    \[
        \bigl(2(h^4+4h^3+6h^2)+1\bigr)^6=O(h^{24})
    \]
    states.
    If $h=0$, all entries and all Markov values vanish, so $H(Q)=0$ and there is just one state.

    For the running time, each state requires at most four mutations and an acyclicity test.
    A tested mutation has entries of magnitude at most $H(Q)+H(Q)^2$, even if it is discarded for exceeding the bound.
    Thus all arithmetic uses integers with $O(\log(h+1))$ bits when $h\geq1$.
    Store the visited six-tuples in a balanced search tree.
    A lookup then takes $O(\log N)$ comparisons for $N$ states, and each comparison takes $O(\log(h+1))$ bit operations.
    Since $N=O(h^{24})$, lookup and exact arithmetic take time polynomial in $\log(h+1)$ per state.
    Thus, with basic integer arithmetic and balanced-tree storage, one obtains $O(h^{24}\log^2(h+1))$ bit operations and $O(h^{24}\log(h+1))$ bits of storage, including the queue and predecessor records.
    In particular, the running time is polynomial in $h$, as claimed.
\end{proof}

\subsection{An example and a conjecture}
The search of Corollary~\ref{thm:algorithm} can also prove mutation-cyclicity when the obstructions of Section~\ref{sec:background} do not apply within the search range.

\begin{example}\label{ex:height3}
    Let $\qtuple{Q}=(-3,-2,2,0,3,-3)$, a connected quiver of height three.
    It has a single proper wiggle class, with Markov value $41$, and $p(Q)^2=225$, so $H(Q)=41$.
    The pair $(41,225)$ is realized by the connected acyclic quiver $(0,0,3,5,1,1)$ in its topological ordering $1<2<3<4$, so $Q$ is not Markov-excluded.
    The component of $Q$ in the mutation graph restricted to quivers of height at most $41$ consists of $238$ quivers, none of which is acyclic, so $Q$ is mutation-cyclic by Corollary~\ref{thm:algorithm}.
    None of these $238$ quivers contains a mutation-cyclic full triangle, is a vortex, or is Markov-excluded, so within this range the obstructions of Section~\ref{sec:background} do not detect $Q$.
    The mutation class of $Q$ contains oriented four-cycles such as $(3,-2,0,0,9,-1)$, with multiplicities $3,9,1,2$ around the cycle.
    These claims are checked by \file{examples.py}.
\end{example}

We do not know whether the term $m(Q)$ in Theorem~\ref{thm:height} is needed.
\begin{conjecture}\label{conj:height}
    Every mutation-acyclic four-vertex quiver $Q$ admits a mutation sequence to an acyclic quiver along which every arrow multiplicity is at most $\htq(Q)$.
\end{conjecture}
If Conjecture~\ref{conj:height} holds, the search of Corollary~\ref{thm:algorithm} can be run with the bound $\htq(Q)$ in place of $H(Q)$, and it has at most $(2h+1)^6=O(h^6)$ states for input height $h\ge1$.
The program \file{conjecture\_evidence.py} gives the following evidence.
\begin{enumerate}[label=(\roman*)]
    \item For every connected four-vertex quiver $Q$ of height at most three, either a search within height $\htq(Q)$ reaches an acyclic quiver, or $Q$ is mutation-cyclic.
    There are $34$, $633$ and $4276$ such quivers of heights one, two and three up to relabeling.
    The mutation-cyclic ones are certified by a mutation-cyclic full triangle, a vortex or a Markov exclusion within three mutations, except for the dreaded torus at height two and four quivers at height three, which are certified by Corollary~\ref{thm:algorithm}.
    Up to relabeling and reversal of all arrows, these four quivers are mutation-equivalent to the quiver of Example~\ref{ex:height3}.
    \item Starting from each connected acyclic quiver with entries in $\{0,1,2\}$, we computed every quiver reachable through quivers of height at most $10$, a total of $94{,}344$ quivers in $320$ regions, and checked that each of them reaches an acyclic quiver inside its region without exceeding its own height.
    \item The same holds for $2000$ random mutation-acyclic quivers of heights between $6$ and $60$, obtained by random mutation sequences from acyclic quivers.
\end{enumerate}
Conjecture~\ref{conj:height} holds for every mutation class that contains an acyclic quiver with at least two arrows between every pair of vertices.
Following \cite[Definition~3.4]{felikson-tumarkin2018}, call a mutation decreasing if it lowers the multiplicity of some arrow and raises none.
By \cite[Corollary~3.5 and Remark~3.6]{felikson-tumarkin2018}, every quiver in such a class reaches an acyclic quiver by finitely many decreasing mutations.
No multiplicity increases along this sequence, so it never exceeds $\htq(Q)$.

\begin{remark}\label{rem:Gx}
    The method of Theorem~\ref{thm:height} does not prove Conjecture~\ref{conj:height}, because the local analysis of Section~\ref{sec:local} uses $m\le T$, and $m(Q)$ can be much larger than $\htq(Q)$.
    For $x\ge1$, the quiver $G_x$ with $\qtuple{G_x}=(x,x,0,1,x,-1)$ is acyclic with topological ordering $1<2<4<3$ and $m(G_x)=x^3+4x^2+x+2$ \claim{S5.Gx-acyclic, S5.Gx-markov}, while $\qtuple{\mu_4G_x}=(x,x,0,x+1,-x,1)$ has height $x+1$ \claim{S5.Gx-mu4}.
    Thus the height grows linearly along this family of mutation-acyclic quivers, while $m$ grows cubically.
    For acyclic quivers $m$ can even be of order $\htq(Q)^4$, since the acyclic quiver with all six entries equal to $h$ in a topological ordering has $m=h^4+4h^3+6h^2$.
\end{remark}

\subsection{Comparing mutation classes}
\begin{corollary}\label{cor:equivalence}
    Mutation equivalence up to vertex relabeling is decidable for pairs of integral four-vertex quivers of which at least one is mutation-acyclic.
    More precisely, there is an algorithm that, given two four-vertex quivers, either decides whether they are mutation equivalent up to relabeling, or reports that both are mutation-cyclic.
\end{corollary}
\begin{proof}
    Apply Corollary~\ref{thm:algorithm} to both inputs.
    If both are mutation-cyclic, report this.
    If exactly one is mutation-acyclic, they cannot be mutation equivalent.
    Otherwise both are mutation-acyclic, and the searches give acyclic representatives $A,B$.
    The original inputs are mutation equivalent up to relabeling exactly when $A,B$ are.

    By Lemma~\ref{lem:acyclicreps}, if $A$ and $B$ are mutation equivalent up to relabeling, then $B$ is reached from $A$, up to relabeling, by mutations at sources and sinks, which only reorient the underlying weighted graph of $A$.
    That graph has at most six edges, hence at most $2^6=64$ orientations.
    Starting at $A$, search all orientations reachable by source or sink mutations.
    At each one, compare with the $24$ relabelings of $B$.
    Finding a match proves equivalence; exhausting this finite search without a match disproves it.
\end{proof}

If both inputs are mutation-cyclic, this argument provides no decision method for their mutation equivalence.
The reduction to acyclic representatives in the second paragraph works in every rank, so only the decision procedure of the first paragraph is specific to rank four.

\section{The bounded classification}\label{sec:correction}
Theorem~2.9 of \cite{AHJL2025} claims that a four-vertex quiver with arrow multiplicities at most two is mutation-cyclic exactly when its mutation class contains a mutation-cyclic full triangle, or the quiver is the dreaded torus $Q_{\mathrm T}$ with $\qtuple{Q_{\mathrm T}}=(-2,1,1,-1,-1,-1)$.
The double-arrow box disproves this criterion.
The computation in that paper already included the box as a mutation-cyclic seed.
Indeed, the proof of \cite[Theorem~2.9]{AHJL2025} lists the box and the dreaded torus as special cases, and arrives at the same counts $534$ and $133$ for connected quivers as Theorem~\ref{thm:census} below.
So the issue is the stated criterion rather than the reported class counts.
Disconnected quivers were not considered in \cite{AHJL2025}, so the second row of the table in Theorem~\ref{thm:census} is new.
We give a corrected criterion with short certificates and recover those counts.

\begin{proposition}\label{prop:box}
    The double-arrow box
    \[
        \qtuple{Q_\square}=(-2,0,2,-2,0,-2)
    \]
    is mutation-cyclic.
    It is not isomorphic to the dreaded torus, and every full triangle in every quiver in its mutation class is mutation-acyclic.
\end{proposition}
\begin{proof}
    The quiver $\mu_1Q_\square$, with $\qtuple{\mu_1Q_\square}=(2,0,-2,-2,4,-2)$, is connected and has a single proper wiggle class, in which $M=0<3$.
    By Lemma~\ref{lem:endpoint}, no connected acyclic quiver has Markov value $0$ in a topological ordering, so $\mu_1Q_\square$ is Markov-excluded, and $Q_\square$ is mutation-cyclic by Lemma~\ref{lem:exclusion}.
    The box is the oriented four-cycle $1\to4\to3\to2\to1$ with all multiplicities two, the case $a=b=2$ of the four-cycles with multiplicities $a,b,a,b$ shown to be mutation-cyclic in \cite[Example~6.11]{warkentin2014}; see also \cite[Example~6.10]{neville2025}.
    The name is from \cite{AHJL2025}.
    The assertion about all full triangles is \cite[discussion after Corollary~5.5]{ervin2024}.
    The box has four underlying edges and the dreaded torus has six, so they are not isomorphic.
\end{proof}
\begin{theorem}[Correction of {\cite[Theorem 2.9]{AHJL2025}}]\label{thm:census}
    Let $Q$ be a four-vertex quiver of height at most two.
    \begin{enumerate}[label=(\roman*)]
        \item $Q$ is mutation-acyclic if and only if a mutation sequence of length at most three takes it to an acyclic quiver.
        Such a sequence can be chosen so that every intermediate quiver has height at most two.
        \item $Q$ is mutation-cyclic if and only if either $Q$ is isomorphic to the dreaded torus, or a sequence of at most two mutations takes $Q$ to a quiver that
        \begin{enumerate}[label=(\alph*)]
            \item contains a mutation-cyclic full triangle, or
            \item is a vortex, or
            \item is connected and Markov-excluded.
        \end{enumerate}
    \end{enumerate}
    The sufficiency of the conditions in (ii) holds for every four-vertex quiver, and the height restriction is needed only for necessity.
    Up to vertex relabeling, with reversal of all arrows not quotiented out, the counts are
    \begin{center}
    \begin{tabular}{lrrr}\toprule
        & Mutation-acyclic & Mutation-cyclic & Total\\\midrule
        Connected & 534 & 133 & 667\\
        Disconnected & 27 & 1 & 28\\\bottomrule
    \end{tabular}
    \end{center}
\end{theorem}
\begin{proof}
    The ``if'' direction of (i) is by definition.
    The ``if'' direction of (ii) holds for any four-vertex quiver.
    Mutation-cyclicity is a property of the mutation class, the dreaded torus is mutation-cyclic because each of its four mutations is a relabeling of $Q_{\mathrm T}$, which contains the directed cycle $1\to3\to2\to1$, and conditions (a), (b), (c) each force mutation-cyclicity by Proposition~\ref{prop:rank3}, Lemma~\ref{lem:vortex} and Lemma~\ref{lem:exclusion} respectively.

    The ``only if'' directions are a finite enumeration.
    For each representative we record either a short path to an acyclic quiver, or a short path to a quiver with one of the obstructions in (ii).
    The supplementary program \file{census.py} performs both searches and writes the certificates to \file{census\_certificates.json}, which contains no interpreter versions or timestamps, so that its checksum is reproducible.
    The program \file{check\_census.py} checks the representative coverage and every certificate using only the independent implementation of mutation, relabeling, acyclicity, full triangles, vortices, proper wiggle classes and Markov values in \file{reference\_q4.py}, which shares no code with the generator.
    For a Markov-exclusion certificate it recomputes every proper wiggle class and its pair $(M,p^2)$, together with the finite table of endpoint pairs used to rule out a match.
    Both programs use exact integer arithmetic.
    The census proof uses these certificates and the obstructions above, independently of Theorem~\ref{thm:height}.

    Enumerate the $5^6=15{,}625$ tuples with entries in $\{-2,-1,0,1,2\}$.
    Replace each by the lexicographically least of its $24$ vertex relabelings.
    This leaves $695$ representatives.

    For each representative, examine all $1+4+4^2+4^3=85$ mutation sequences of length at most three.
    Discard a sequence if it leaves height two.
    There are $561$ representatives for which a remaining sequence ends at an acyclic quiver.
    Their shortest lengths are
    \[
    \begin{array}{c|rrrr}
        \text{length}&0&1&2&3\\\hline
        \text{number}&425&118&16&2
    \end{array}
    \]

    Of the $134$ remaining representatives, one is disconnected.
    It is $(-2,0,2,0,-2,0)$, an oriented triangle of multiplicities $(2,2,2)$ and an isolated vertex.
    It satisfies (a) with no mutation.
    One connected representative is the dreaded torus.
    For each of the other $132$, a sequence of at most two mutations reaches a connected quiver satisfying (a), (b) or (c).
    The first two conditions are checked directly; for (c), use the finite computation of Lemma~\ref{lem:exclusion} for each proper class.
    For each representative we take a shortest certificate, and among certificates of the same length we prefer (a), then (b), then (c).
    With this rule, (a), (b) and (c) are used for $45$, $40$, and $47$ representatives, respectively.
    For $127$ representatives no mutation is needed; four require one mutation and one requires two.
    Thus every representative not found by the acyclic search satisfies one of the conditions in (ii), which proves both ``only if'' directions.
    Separating the representatives by connectedness gives the stated counts.
\end{proof}

\begin{remark}\label{rem:censussharp}
    The bound three in Theorem~\ref{thm:census}(i) is attained by
    \[
        (-2,-1,2,-1,-1,-1)\xrightarrow{\mword{4,1,3}}\text{acyclic},
        \qquad
        (-2,1,1,-2,1,-1)\xrightarrow{\mword{3,2,4}}\text{acyclic}.
    \]
    Neither quiver has an acyclic endpoint after at most two mutations, even if the intermediate height is unrestricted; this is checked by the finite list of words of length at most two.
    For the negative certificate system in part~(ii), the representative
    \[
        \qtuple{R}=(-2,-1,2,-1,-1,-2)
    \]
    requires two mutations.
    Indeed, $\mu_{\mword{1,3}}R$, whose tuple is $(3,-1,-2,1,3,0)$, contains the cyclic triangle on $\{1,2,4\}$ with multiplicities $(3,3,2)$ and Markov constant $4$.
    Neither $R$ nor any of its four one-mutation neighbors is a vortex or contains a mutation-cyclic full triangle.
    Each has a single proper wiggle class with $(M,p^2)=(3,1)$, a pair realized by the connected acyclic tuple $(1,0,0,1,0,1)$, so none is Markov-excluded.
    These examples, including the absence of shorter certificates, are checked by \file{check\_census.py}.
\end{remark}

\appendix
\section{Proofs of the local lemmas}\label{app:lemmas}
Throughout the appendices, retain the setup of Section~\ref{sec:local}, so that
$Q_*$ is connected and mutation-acyclic, $m=m(Q_*)$, $T\ge m$, and $F=T+1$.
Every quiver considered belongs to the mutation class of $Q_*$, so
\[
 3\le m\le T,\qquad 2|p|<m
\]
by \eqref{eq:prefixscalar}, where $m$ is the Markov invariant of the totally proper class.
For each local lemma, relabel the specified vertices as $1,2$.

Every calculation in these appendices is stated as a claim in the file \file{proof\_claims.py} of the supplement, and is cited here by its label in typewriter font.
The program \file{check\_proofs.py} verifies each claim exactly and for all integer values of its parameters.
Polynomial identities are checked by expansion.
Inequalities and contradictions are checked with explicit certificates, which write the goal as a nonnegative combination of products of the hypotheses and of squares, modulo the stated equations, sometimes after splitting on whether an integer quantity is at least $1$ or at most $0$.
Mutation formulas are checked by symbolic mutation, after certifying the sign of every entry that the mutation formula \eqref{eq:mutation} needs.
Statements about proper wiggle classes, acyclicity, vortices and oriented triangles depend only on the signs of the entries, and are checked by enumerating all sign patterns allowed by the hypotheses.
Each claim lists the hypotheses it uses, including the standing bounds $F\ge4$, $3\le m\le F-1$ and $2|p|<m$ where needed.
The case analysis itself, that is, the choice of cases and the argument that they cover all possibilities, is written out below and is not checked by machine.
The program \file{bruteforce\_local.py} tests it independently, by enumerating all excursions for $T\le6$ using only the constraints that the proofs use.
The verification relies on the exact arithmetic of \file{polynomial.py}, the certificate checker in \file{certificates.py}, the symbolic mutation in \file{symbolic.py}, and the functions of \file{quiver4.py} that compute proper wiggle classes, acyclicity and vortices.
The hand calculations behind the claims, in the form of an earlier version of these appendices, are reproduced with the same labels in the file \file{APPENDIX\_CALCULATIONS.md} of the supplement.

We first collect two observations used in several proofs.
\begin{lemma}\label{lem:appendixorder}
    Let $P$ be mutation-acyclic, with
    $\qtuple{P}=(a,b,-c,d,E,f)$.
    Assume $b,c,f>0$ and $a\ge0$, and assume moreover that one of the following holds.
    \begin{enumerate}[label=(\roman*)]
        \item $d\ge0$; the entry $E$ may have any sign.
        \item $d<0$, $a=0$, and $E\le0$.
    \end{enumerate}
    Then $P$ has exactly one proper wiggle class, and this class contains
    $\omega_0=\corder{1,2,3,4}$.
    In particular, $m=M_{\omega_0}(P)$.
\end{lemma}
\begin{proof}
    The complete directed triangle on $\{1,3,4\}$ restricts every proper class to
    $\corder{1,2,3,4}$, $\corder{1,3,2,4}$, and $\corder{1,3,4,2}$, and in each sign case the remaining entries either exclude the last two orders or join them to $\omega_0$ by wiggles.
    The proper wiggle classes depend only on the signs of the entries, and for every sign pattern allowed by (i) or (ii) there is at most one proper wiggle class, which contains $\omega_0$ \claim{A0.ord1, A0.ord2}.
    Theorem~\ref{thm:FN} supplies a totally proper class, so it is that class.
\end{proof}

\begin{lemma}\label{lem:productbound}
    Let $F,r,s$ be positive integers satisfying
    $1\le r,s\le F-1$ and $rs=2F$.
    Then $r,s\ge3$ and
    \begin{equation}\label{eq:productbound}
        3(r^2+s^2)\le2F^2-5F+36.
    \end{equation}
\end{lemma}
\begin{proof}
    If $r\le2$, then $s=2F/r\ge F$, contradicting $s\le F-1$, and symmetrically for $s$ \claim{A0.pb1}.
    The inequality follows from $(r-3)(F-1-r)\ge0$, $(s-3)(F-1-s)\ge0$ and $(r-3)(s-3)\ge0$ \claim{A0.pb2}.
\end{proof}

\subsection{Proof of Lemma~\ref{lem:zero}}\label{app:zero}
Because vertices $1,2$ are not joined, mutation at either one leaves all entries incident to the other unchanged.
The low neighbor $\mu_1P$ therefore bounds every entry except $b_{34}$ by $T$.
Exchange $3,4$ and reverse arrows as needed to write
\[
    \qtuple{P}=(0,b,-c,d,E,F),\qquad b,c\ge1,
    \qquad b,c,|d|,|E|\le T.
\]
(The signs of $b,c$ are forced, because mutation at $1$ lowers the $b_{34}$ entry only if its arrows point $4\to1\to3$, in which case it becomes $F-bc$ \claim{A1.mu1P}.)
The case $d<0<E$ is impossible, because mutation at $2$ would increase $b_{34}$ to $F+(-d)E>F$ \claim{A1.mu2-dneg-Epos, A1.dneg-Epos}.
Lemma~\ref{lem:appendixorder} gives \claim{A1.markov}
\begin{equation}\label{eq:zeroMarkov}
    m=b^2+c^2+d^2+E^2+F^2-bcF+dEF.
\end{equation}

We first show that from $P$, and from every quiver of the same form, the mutations at $3$ and $4$ leave height $F$.
This is applied to the second high quiver below.
Mutation at $3$ would create the entries $bF-c$ and $E+F[d]_+$ \claim{A1.mu3-dpos, A1.mu3-dneg}.
If its height were at most $F$, then $b=1$ and $d\le1$, with $E\le0$ when $d=1$ \claim{A1.mu3-b1, A1.mu3-d1, A1.mu3-E}.
With $b=1$, each of the cases $d=0$, $d<0$ (where $E\le0$, since $d<0<E$ is excluded) and $d=1$ gives $m>F$ \claim{A1.id-dnonpos, A1.dzero, A1.dneg-Enonpos, A1.id-d1, A1.d1}, contradicting $m\le T$.
Mutation at $4$ is ruled out by the same argument after reversing arrows and exchanging $3,4$, which sends $(b,c,d,E)$ to $(c,b,-E,-d)$ and preserves the form of the calculation \claim{A1.sym34}.

If $\mu_2P$ is high, only its $b_{34}$ entry can have magnitude $F$.
That entry stays $F$ unless $d>0>E$, in which case it becomes $F+dE$ \claim{A1.mu2-dpos-Eneg, A1.mu2-dpos-Epos, A1.mu2-dneg-Eneg}.
So it either stays $F$ or becomes $-F$, and the latter requires $d>0$, $E=-e<0$ and $de=2F$, where $d,e\le F-1$.
Then $d,e\ge3$ \claim{A1.minusF-d3, A1.minusF-e3}, and the rectangle bound \eqref{eq:rectangle} and Lemma~\ref{lem:productbound} give $3m\le42-8F-F^2<0$, a contradiction \claim{A1.minusF}.
Thus a second high quiver has the form $(0,b,-c,-d,-E,F)$.
It satisfies the same sign restriction and the same Markov bound, so the preceding argument again forbids $3,4$.
The next reduced mutation must be at $1$.
It gives the entry $F-bc$ at $b_{34}$, exactly as in the low quiver $\mu_1P$, and all its other entries are low \claim{A1.second-dneg-Eneg, A1.second-dpos-Epos}.
This proves the first assertion.

For the second assertion, the two mutations commute and are involutions, so there are at most four labeled quivers to consider.
Only a low endpoint requiring both mutations needs attention.
If a one-mutation intermediate quiver $P$ has height greater than $T$, its two neighbors are the low start and low endpoint.
Such a temporary corner need not have height at most $F$; the bound $m\le T$ still holds throughout the mutation class.
The two low neighbors force every incident entry to be low, and both lower the remaining entry.
Thus normalize
\[
    \qtuple{P}=(0,b,-c,d,-e,f),\qquad
    1\le b,c,d,e\le T,\qquad f>T.
\]
The two low neighbors have $b_{34}$ entries $f-bc$ and $f-de$, and the other corner of the commuting square has entry $f-bc-de$ \claim{A1.sq-paths, A1.sq-paths2}.
Its upper bound follows from $f-bc\le T$; we must prove
\begin{equation}\label{eq:zerosquare}
    f-bc-de\ge-T.
\end{equation}
Lemma~\ref{lem:appendixorder} also gives \claim{A1.sq-markov, A1.sq-pf}
\[
    m=b^2+c^2+d^2+e^2+f^2-f(bc+de),\qquad p=be-cd.
\]
Assume for a contradiction that $bc+de>f+T$.

If all four of $b,c,d,e$ are at least two, then $4m\le9$, contradicting $m\ge3$; the two claims \claim{A1.sq-all2-rs, A1.sq-all2-sr} treat $de\le bc$ and $bc\le de$.

Otherwise some entry is $1$.
The transformations
\[
    (b,c,d,e)\longmapsto(c,b,e,d),\qquad
    (b,c,d,e)\longmapsto(d,e,b,c)
\]
are realized by reversing all arrows and exchanging $3,4$, and by exchanging $1,2$ \claim{A1.sq-sym1-relabel, A1.sq-sym2-relabel}.
They preserve all the hypotheses and $m$, and change only the sign of $p$ \claim{A1.sq-sym1, A1.sq-sym2}.
They can put a unit entry at $b=1$.
Then neither $d$ nor $e$ can be $1$ \claim{A1.u-d1, A1.u-e1}.
If also $c=1$, then $f=F$ and $de=2F-1$ \claim{A1.u-c1-f, A1.u-c1-dehi, A1.u-c1-delo}, so $d,e$ are odd and at least three \claim{A1.u-c1-d2, A1.u-c1-e2}, and then $m<3$ \claim{A1.u-c1-main}.
Assume from now on that $c\ge2$, so that $c,d,e\ge2$.
We use one common estimate \claim{A1.u-common}:
\begin{equation}\label{eq:zeroCommon}
    m\le(d-e)^2+4.
\end{equation}
It remains to compare $c$ with $d,e$.
\begin{enumerate}[label=\arabic*.]
    \item If $c<\max(d,e)$, then $m\le2$ \claim{A1.u-case1-d, A1.u-case1-e}.
    \item If $e\le d\le c$, then $2|p|=2(cd-e)\ge(d-e)^2+4\ge m$ by \eqref{eq:zeroCommon}, contradicting $2|p|<m$ \claim{A1.u-case2}.
    \item If $d<e=c$, then $m\le2|p|$, contradicting $2|p|<m$ \claim{A1.u-case3-d2, A1.u-case3-d3}.
    \item If $d<e<c$, then $f<de<2(cd-e)<m\le T<f$ \claim{A1.u-case4}.
\end{enumerate}
This proves \eqref{eq:zerosquare}.
Thus whenever one intermediate corner has height greater than $T$, the other is low.
The alternative two-mutation path has exactly the same labeled endpoint and stays low, proving the second assertion. \qed

\subsection{Proof of Lemma~\ref{lem:unitroute}}\label{app:unit}
Put $N_0=C$, $N_1=\mu_1C$, $N_2=\mu_{\mword{1,2}}C$, $N_3=\mu_{\mword{1,2,1}}C$, $N_4=\mu_{\mword{1,2,1,2}}C$.
By the pentagon relation \eqref{eq:pentagon}, there are at most ten labeled quivers reachable using only $1,2$, namely $N_r,N_r^{\pi_{12}}$.
Mutation at $1$ or $2$ gives the following transitions:
\begin{center}\begin{tabular}{c|cc}
          & $1$ & $2$\\\hline
    $N_0$ & $N_1$ & $N_4^{\pi_{12}}$\\
    $N_1$ & $N_0$ & $N_2$\\
    $N_2$ & $N_3$ & $N_1$\\
    $N_3$ & $N_2$ & $N_4$\\
    $N_4$ & $N_0^{\pi_{12}}$ & $N_3$
\end{tabular}\end{center}
The transitions from $N_r^{\pi_{12}}$ follow by swapping the mutation vertices.
It is enough to reach the endpoint up to $\pi_{12}$.
The endpoints $N_0,N_1,N_4$ require at most one mutation.
For $N_2$, use $\mword{1,2}$ if $\htq(N_1)\le K$, and $\mword{2,1,2}$ otherwise.
For $N_3$, use $\mword{2,1}$ if $\htq(N_4)\le K$, and $\mword{1,2,1}$ otherwise.
The next calculation proves that each detour stays within height $K$.

It remains to justify the detour when a direct intermediate quiver $P$ exceeds height $K$, but both neighbors $\mu_1P,\mu_2P$ have height at most $K$.
Every entry touching vertex $1$ or $2$ is bounded by $K$, because one of these mutations simply reverses that entry.
Thus only the $b_{34}$ entry can exceed $K$.
Reverse arrows and exchange $3,4$ if necessary to make $b_{12}(P)=1$ and $b_{34}(P)=f>K$.
For mutation at $1$ to reduce $f$, its incident arrows must be $4\to1\to3$, and for mutation at $2$ to reduce $f$, they must be $4\to2\to3$.
Therefore
\[
    \qtuple{P}=(1,b,-c,d,-e,f),\qquad b,c,d,e\ge1.
\]
The tuples of the two low neighbors are \claim{A2.neighbours1, A2.neighbours2}
\[
    (-1,-b,c,d,-c-e,f-bc),\qquad (-1,b+d,-c,-d,e,f-de).
\]
They give $b+d\le K$, $c+e\le K$, and $|f-bc|,|f-de|\le K$.
Lemma~\ref{lem:appendixorder} identifies the totally proper class, so \claim{A2.markov, A2.pf}
\[
    D=bc+de+cd,\quad m=1+b^2+c^2+d^2+e^2+bd+ce+f^2-fD,\quad p=f+be-cd.
\]
Starting at $\mu_1P$, the alternative path $\mword{2,1,2}$ ends at $(\mu_2P)^{\pi_{12}}$, and its two intermediate tuples are \claim{A2.detour}
\[
    (1,-b,-e,-d,c+e,f-D),\quad(-1,b,e,-b-d,c,f-D).
\]
All their entries are bounded by $K$ except possibly $f-D$.
Its upper bound follows from $f-D\le f-bc\le K$.
We must prove $f-D\ge-K$.
Suppose instead $f+K<D$.
If $c,d\ge2$, then $m\le-1$, contradicting $m\ge3$; the two claims \claim{A2.cd2-SU, A2.cd2-US} treat $b+d\ge c+e$ and $c+e\ge b+d$.
If $d=1$, then $b,c\ge2$ \claim{A2.d1-b1, A2.d1-c1}, and $m-2p\le-5$, contradicting $2|p|<m$; with $S=b+1$ and $U=c+e$, the three claims \claim{A2.d1-SU, A2.d1-US-i, A2.d1-US-ii} treat $S\ge U$, and $U\ge S$ with $c\le S-1$ or $c\ge S-1$.
The case $c=1$ follows by the symmetry $(b,c,d,e)\mapsto(e,d,c,b)$, which is realized by reversing all arrows and relabeling by $(1\,2)(3\,4)$, and which preserves the hypotheses \claim{A2.sym-relabel, A2.sym}.
Thus $f-D\ge-K$, proving the height bound. \qed

\subsection{Proof of Lemma~\ref{lem:strong}}\label{app:strong}
Reverse all arrows if necessary so that $a=b_{12}(P)\ge2$.
All entries incident to $1$ have magnitude less than $F$, because mutation at $1$ only reverses them.
The entries incident to $2$ also have magnitude less than $F$.
Indeed, for $k=3,4$, mutation at $1$ replaces $b_{2k}(P)$ by $b_{2k}(P)-a[-b_{1k}(P)]_+$.
Thus $b_{2k}(P)=-F$ could not become low.
If $b_{2k}(P)=F$ becomes low, then $b_{1k}(P)<0$, and mutation at $2$ would create the entry $b_{1k}(P)+aF>F$, contrary to its height bound \claim{A3.b23-mu1-ypos, A3.b23-mu1-yneg, A3.b23m-mu1, A3.b23m-mu1-ypos, A3.b23-mu2, A3.b23-big}; these claims treat $k=3$, and $k=4$ follows by exchanging $3,4$.
Only the entry $b_{34}(P)$ can therefore have magnitude $F$.
Exchange $3,4$ if necessary to make that entry positive.
For mutation at $1$ to lower it, its two arrows must point $4\to1\to3$.
Consequently
\[
    \qtuple{P}=(a,b,-c,d,E,F),\qquad a\ge2,\quad
    b,c\ge1,\quad a,b,c,|d|,|E|\le F-1.
\]
The two neighbor bounds also give
\[
    |E-ac|\le F-1,\qquad b+ad\le F\quad\text{if }d\ge0
\]
\claim{A3.nb1, A3.nb2}.
Its Pfaffian is $p=aF-bE-cd$ \claim{A3.pf}.

We first show $d\ge0$ and $E\ge c$.
If $d<0<E$, mutation at $2$ creates the entry $F+(-d)E>F$ \claim{A3.mu2-dneg-Epos, A3.dneg-Epos}.
If $d<0$ and $E\le0$, then $p\ge aF\ge2F$, contradicting \eqref{eq:prefixscalar} \claim{A3.dneg-Enonpos}.
The same contradiction excludes $d=0,E<0$ \claim{A3.d0-Eneg}.
In all remaining cases, Lemma~\ref{lem:appendixorder} gives $m=M_{\omega_0}(P)$.
If $d\ge0$ and $0\le E<c$, put $x=b+(a-1)d$ and $y=c-E$, so that $1\le x,y\le F-1$.
The identity
\[
    m=a^2+x^2+y^2-Fxy-(a-1)F^2+Fp
    -(a-2)\{d(b+ad)+cE+FdE\}
\]
\claim{A3.identity1}, the rectangle bound \eqref{eq:rectangle} and $2p<m\le F-1$ give $2m<6$, contradicting $m\ge3$ \claim{A3.case-Esmall-d1, A3.case-Esmall-d0}.
If $d>0$ and $E=-e<0$, then $ac+e\le F-1$ and $b+(a-1)d\le F-1$, and a similar estimate \claim{A3.identity2, A3.case-Eneg} again gives $2m<6$.
Thus $d\ge0$ and $E\ge c$ as claimed.

Write $\beta=b+ad$ and $\gamma=aE-c$.
Mutation at $2$ gives \claim{A3.mu2-rest, A3.gamma-ge-E}
\begin{equation}\label{eq:strongnext}
    \qtuple{\mu_2P}=(-a,\beta,\gamma,-d,-E,F),\qquad \beta\ge1,\quad \gamma\ge E\ge1.
\end{equation}
Mutation at $3$ would create the entry $\gamma+\beta F>F$ \claim{A3.mu2-then3, A3.mu2-then3-big}.
Mutation at $4$ would create an entry of magnitude $d+EF>F$, unless $d=0,E=c=1$ \claim{A3.mu2-then4, A3.mu2-then4-d, A3.mu2-then4-E}.
In this exception the unique proper class identified above gives $m=a^2+b^2+F^2-a-bF+2>F$, contradicting $m\le T$ \claim{A3.exception-m, A3.exception}.
A reduced continuation must therefore mutate at $1$.

After this mutation, $b_{12}$ is positive, $b_{34}=F$, and the two columns joining $1,2$ to $3,4$ are $(-\beta,a\beta-d)$ and $(-\gamma,a\gamma-E)$ \claim{A3.mu1-after}.
Each has the form $(-X,Y)$ with $1\le X\le Y$ \claim{A3.form-beta, A3.form-gamma}.
From any quiver of this form, mutation at $3$ creates the $b_{24}$ entry $Y_4+Y_3F>F$, and mutation at $4$ creates the $b_{13}$ entry $-X_3-X_4F<-F$ \claim{A3.W-mu3, A3.W-mu3-big, A3.W-mu4, A3.W-mu4-big}.
Only $1,2$ can be used.
To see that the same argument keeps applying, mutate at $2$ and exchange labels $1,2$.
A column $(-X,Y)$ becomes $(-Y,aY-X)$ with $aY-X\ge Y\ge1$ \claim{A3.W-mu2, A3.W-form3, A3.W-form4}.
The entry $b_{12}$ is again positive, and $b_{34}$ remains $F$.
Thus every reduced continuation is forced to alternate at $1,2$, and a forbidden choice at $3$ or $4$ exceeds height $F$.
The unchanged entry $b_{34}=F$ prevents any return to a low quiver. \qed

\section{Proof of Proposition~\ref{prop:unit}}\label{app:signs}
We keep the setup of Section~\ref{sec:local} and of Proposition~\ref{prop:unit}.
In particular, the route minimizes \eqref{eq:routecost}, its maximum height is $F$, and $C$ is the low quiver preceding a high segment of this route; the first two mutation vertices are $1,2$, and $b_{12}(C)=1$.
Part (a) is Lemma~\ref{lem:unitroute} with $K=T$.
From now on the segment uses a third vertex.
No quiver before the endpoint of a minimizing route is acyclic.
Otherwise, truncating there would decrease its maximum height, its number of visits at that height, or its length, in that order.
In particular, the low predecessor $C$ is not acyclic.
We use this observation to discard acyclic sign cases immediately.

\subsection{The prefix before the third vertex is \texorpdfstring{$\mword{1,2}$}{[1,2]}}\label{app:prefix}
Before the first mutation at a third vertex, the route alternates at $1,2$ and reaches a quiver of height at most $F$ from the low quiver $C$.
By Lemma~\ref{lem:unitroute} with $K=F$, that quiver is reached, up to the swap of $1,2$, by at most three mutations within height $F$.
If the alternating prefix had length four or more, replacing it by this shorter route and transporting the tail by \eqref{eq:tailtransport} would remove at least one high quiver (the prefix consists of high quivers, and the shorter route contains fewer of them), contradicting the minimizing-route hypothesis in Proposition~\ref{prop:unit}.
Hence the prefix is $\mword{1,2}$ or $\mword{1,2,1}$.
We now exclude $\mword{1,2,1}$.
The pentagon relation gives $\mu_{\mword{2,1}}C=(\mu_{\mword{1,2,1}}C)^{\pi_{12}}$.
If $\mu_2C$ had height at most $F$, replacing $\mword{1,2,1}$ by $\mword{2,1}$ would remove a high quiver.
In each replacement, the old prefix has three high quivers, while the new prefix has at most two, and transporting the tail preserves all its heights.
Thus minimality forces $\htq(\mu_2C)>F$.
The next calculation shows that this is impossible.

Write the tuple of the low starting quiver as $\qtuple{C}=(1,x,c,d,e,f)$.
Because $\mu_{\mword{2,1}}C$ is a relabeling of the third high quiver, it has height at most $F$.
Its last mutation reverses every entry touching $1$, and the entries touching $2$ of $\mu_2C$ are reversals of entries of $C$.
Hence only the $b_{34}$ entry of $\mu_2C$ can exceed $F$.
Exchange $3,4$ to make it less than $-F$.
To decrease $f$ at vertex $2$, we must have $d>0$ and $e=-E<0$, and then $b_{34}(\mu_2C)=f-dE<-F$, while the $b_{13},b_{14}$ entries of $\mu_2C$ are $x+d,c$ \claim{B1.mu2C}.
For mutation at $1$ to raise its $b_{34}$ entry back into $[-F,F]$, we therefore need $x+d<0<c$.
Write $x=-b$, thus $b>d\ge1$ and $c,E\ge1$.
The first high quiver $\mu_1C$ can now reach height $F$ only in its $b_{34}$ entry, which is $f+bc$ \claim{B1.mu1C}.
Since $f\ge-(F-1)$ and $bc>0$, this entry cannot be $-F$.
Consequently
\[
    \qtuple{C}=(1,-b,c,d,-E,F-bc),\quad b,c,d,E\ge1,\quad d<b.
\]
The height bounds on $C$, on $\mu_2C$, and on the quivers after $\mword{1},\mword{1,2},\mword{1,2,1}$ \claim{B1.prefix} give
\[
\begin{gathered}
    bc\le2F-1,\qquad 2F<bc+dE,\qquad c+E\le F,\\
    \text{and either}\quad c+E=F\quad\text{or}\quad d(c+E)=2F.
\end{gathered}
\]
The directed triangle $1\to2\to3\to1$ fixes the cyclic order $\corder{1,2,3}$, and the path $1\to4\to2$ in the complete triple $\{1,2,4\}$ places $4$ between $1,2$.
Thus $C$ has the unique proper wiggle class represented by $\corder{1,4,2,3}$ \claim{B1.order}, and \claim{B1.markov, B1.pf}
\[
    m=1+b^2+c^2+d^2+E^2+cE-bd+F^2-F(bc+dE),\quad p=F-b(c+E)+cd.
\]
If $c+E=F$, then $b\ge3$ \claim{B1.ce-b2}, and the identity $b(-2p-F)=F(2b^2-3b-4d)+2d(2F-bc)$ \claim{B1.ce-identity} shows $-2p>F$, contradicting $2|p|<m\le F-1$ \claim{B1.ce}.
Otherwise $d(c+E)=2F$, so that $d,c+E\ge3$ and $F\ge5$ \claim{B1.dl-d3, B1.dl-L3, B1.dl-F5}, and Lemma~\ref{lem:productbound} gives $3m\le-F^2-5F+36-3cE<9$, contradicting $m\ge3$ \claim{B1.dl}.
Hence the first mutation at a third vertex follows the sequence $\mword{1,2}$.

\subsection{The sign cases}
Write the tuple of the low starting quiver as $\qtuple{C}=(1,x,c,d,e,f)$, and relabel the third vertex as $3$.
The quivers after $\mword{1}$ and $\mword{1,2}$ have height $F$, and the quiver after $\mword{1,2,3}$ has height at most $F$.
We write $P$ and $Q$ for the quivers after $\mword{1}$ and $\mword{1,2}$, so that $Q=\mu_2P$ is a local name in this appendix, and we refer to the height bounds on $P$, $Q$ and $\mu_3Q$ as the first, second and third height bounds.
There are three cases, according to the signs of $x$ and $d$.
The outcome is that only the quivers $A_c$ and $I_{t,u}$ of \eqref{eq:AI} survive.

\subsubsection{Case $x=-b<0<d$}\label{app:case1}
Here $1\le b,d\le F-1$, and all entries of $C$ have absolute value at most $F-1$.

First suppose $c>0$.
Mutation at $1$ gives $(-1,b,-c,d-b,e,f+bc)$ \claim{B21.P-cpos}, so the first high quiver can reach height $F$ only through its $b_{34}$ entry, and $f+bc=F$.
We consider the three possible comparisons between $b$ and $d$.

If $d<b$, put $r=b-d\ge1$.
For $e\ge0$, the $b_{34}$ entry of $Q$ is $F+re$, and for $e\le0$, the $b_{24}$ entry of $\mu_3Q$ is $rF-e$ \claim{B21.a-Q-epos, B21.a-Q-eneg}.
The second and third height bounds therefore give $e=0$ \claim{B21.a-epos, B21.a-eneg}, then $r=1$ \claim{B21.a-r} and $d=1$ \claim{B21.a-d}.
Thus $b=2$ and
\[
    \qtuple{C}=(1,-2,c,1,0,F-2c)=\qtuple{A_c}.
\]

If $b<d$, put $r=d-b\ge1$.
When $e\ge0$, the Pfaffian is $p=F+rc+be\ge F$, contrary to \eqref{eq:prefixscalar} \claim{B21.b-pf, B21.b-epos}.
Hence $e=-E<0$.
Now $\qtuple{Q}=(1,b,-c-E,-r,E,F-rE)$ \claim{B21.b-Q}, and $\mu_3Q$ has $b_{12}=1+br$ and $b_{24}=E-r[rE-F]_+$ \claim{B21.b-mu3Q-a, B21.b-mu3Q-b}.
Because $Q$ has height $F$, either $c+E=F$ or $rE=2F$.
The latter forces $r\ge3$, and then $b_{24}(\mu_3Q)<-F$ \claim{B21.b-rE2, B21.b-rE3}.
Consequently $c+E=F$.
The directed triangle $1\to2\to3\to1$ and the path $1\to4\to2$ again leave the single proper wiggle class of $\corder{1,4,2,3}$ \claim{B1.order}, and evaluation in this ordering gives \claim{B21.markov-1423}
\begin{equation}\label{eq:positive-order-markov}
    m=1+b^2+c^2+d^2+E^2+f^2+cE+bcf-bd-dEf-bcdE.
\end{equation}
Using $d=b+r$, $c+E=F$, and $f=F-bc$, this becomes \claim{B21.b-reduced, B21.b-reduced-p}
\[
    m=1+b^2+br+r^2-(b-2)F^2-rFE-cE,\qquad p=(r+1)c-(b-1)E.
\]
If $b\ge3$, then $m<3$ \claim{B21.b-b3}, and if $b=2$, then $m\le1$ \claim{B21.b-b2-identity, B21.b-b2}.
We must therefore have $b=1$.
Now $p=(r+1)c>0$, and $2p<m\le F-1=c+E-1$ gives $E\ge(2r+1)c+2$ \claim{B21.b-b1-E}.
But $f=E$, so the full triangle on $\{2,3,4\}$ is an oriented cycle with multiplicities $(r+1,E,E)$ \claim{B21.b-b1-triangle}, all at least two, and its Markov constant $(r+1)^2-(r-1)E^2$ is at most $4$ \claim{B21.b-b1-markov3}.
Proposition~\ref{prop:rank3} excludes this triangle.
Thus $b<d$ is impossible when $c>0$.

If $b=d$, the Pfaffian is $p=F+be$, so again $e=-E<0$ \claim{B21.c-pf, B21.c-epos}.
Here $\qtuple{Q}=(1,b,-c-E,0,E,F)$ and $b_{14}(\mu_3Q)=bF-c-E$ \claim{B21.c-Q}, which forces $b\le2$ \claim{B21.c-b3}.
The same ordering argument identifies the transported Markov value with \eqref{eq:positive-order-markov}.
At $b=1$, we get $m>F$ \claim{B21.c-b1-m, B21.c-b1}.
At $b=2$, the height bound on $Q$ gives $c+E\le F$, while the $b_{14}$ entry of $\mu_3Q$ gives $c+E\ge F$.
Thus $c+E=F$, and then $m=5-cE$ forces $c+E\le3<F$ \claim{B21.c-b2-m, B21.c-b2}.

Now suppose $c\le0$.
Mutation at $1$ gives $(-1,b,-c,d-b,e+c,f)$ \claim{B21.n-P}, whose only entry that can have magnitude $F$ is $e+c$.
Hence $c=-t<0$, $e=-v<0$, and $t+v=F$, with $t,v\ge1$.
If $f<0$, vertex $4$ is a source and $C$ is a vortex \claim{B21.n-vortex}, which Lemma~\ref{lem:vortex} excludes.
We may assume $0\le f\le F-1$.

If $d<b$, put $r=b-d\ge1$.
Then $\qtuple{Q}=(1,d,-v,r,F,f)$ and $b_{24}(\mu_3Q)=F+rf$ \claim{B21.n-dlt}.
The third height bound gives $f=0$ \claim{B21.n-dlt-f}, but then $|p|=bv+td\ge F$, contradicting \eqref{eq:prefixscalar} \claim{B21.n-dlt-p}.

If $d=b$, then $p=f-bF$ \claim{B21.n-deq-p}, so $b=1$ \claim{B21.n-deq-b2}.
The unique proper wiggle class contains $\corder{1,2,3,4}$ \claim{B21.n-deq-order0, B21.n-deq-order1}, and evaluation in this class gives $m=2+t^2+v^2+f^2+tv-vf>F$, again impossible \claim{B21.n-deq-m, B21.n-deq-b1}.

Finally, if $d>b$, put $r=d-b\ge1$.
Then $\qtuple{Q}=(1,b,-v,-r,F,f-rF)$ and $b_{34}(\mu_3Q)=rF-f$ \claim{B21.n-dgt}.
Since $f\le F-1$, the third height bound forces $r=1$ \claim{B21.n-dgt-r}.
Put $u=F-f\ge1$.
Then $|p|=(b-1)F+t+u$ \claim{B21.n-dgt-p}, so $b=1$ \claim{B21.n-dgt-b2}, $d=2$, and $2(t+u)<m\le F-1$ \claim{B21.n-dgt-b1}.
The starting quiver therefore has tuple
\[
    \qtuple{C}=(1,-1,-t,2,t-F,F-u)=\qtuple{I_{t,u}},\qquad 1\le t,u\le F-1.
\]
These cases leave only $A_c$ and $I_{t,u}$.

\subsubsection{Case $x=-b<0$, $d=-D\le0$}\label{app:case2}
Here $1\le b\le T$ and $0\le D\le T$.
Put
\[
    h=e-[-c]_+,\qquad g=f+b[c]_+,\qquad z=g+(b+D)[h]_+.
\]
In each of the eight sign cases of $c$, $h$ and $z$, which are the claims \claim{B22.paths-+++} through \claim{B22.paths----},
\[
    \qtuple{P}=(-1,b,-c,-D-b,h,g),\qquad\qtuple{Q}=(1,-D,-c-[-h]_+,D+b,-h,z),
\]
and two entries after the third mutation are
\[
    b_{12}(\mu_3Q)=1-D(D+b),\qquad b_{24}(\mu_3Q)=-h+(D+b)[z]_+.
\]
All entries of $C$ have absolute value at most $T$.
Consequently $P$ can reach height $F$ only through
\[
    h=-F,\qquad g=F,\qquad\text{or}\qquad b+D=F.
\]
We exclude these possibilities in that order, since they need not be disjoint.

If $h=-F$, then $c,e<0$ and $g=z=f$.
The third height bound forces $f\le0$ \claim{B22.hF-f}, and then $C$ is acyclic \claim{B22.hF-acyclic}, contrary to its nonacyclicity.

If $g=F$, then $c>0$, and $z\le F$ forces $e\le0$ \claim{B22.gF-e}.
The third height bound forces $b=1$ and $D=0$ \claim{B22.gF-bD}, and then $e=0$ \claim{B22.gF-e0}.
Hence $C=(1,-1,c,0,0,F-c)$ with $0<c<F$, which is acyclic \claim{B22.gF-acyclic}.

It remains to consider $b+D=F$.
Since $b\le T$, we have $D\ge1$.
The third $b_{12}$ bound forces $D=1$ and $b=T$ \claim{B22.bD-D}, and the third $b_{24}$ bound forces $z\le1$ \claim{B22.bD-z}.
The Pfaffian is now $p=f+Te-c$ \claim{B22.bD-pf}.
If $c>0$, then $z\le1$ forces $e\le0$ \claim{B22.bD-cpos-e}, and then $p\le-T$, contrary to~\eqref{eq:prefixscalar} \claim{B22.bD-cpos-p}.

We may therefore put $c=-t$, where $t\ge0$, so that $z=f+F[e-t]_+\ge f$.
If $f>0$, the inequality $z\le1$ forces $f=1$ and $e\le t$ \claim{B22.bD-cneg-f2, B22.bD-cneg-et}.
The third $b_{24}$ entry is then $F+t-e$, so $e\ge t$ as well \claim{B22.bD-cneg-b24, B22.bD-cneg-ete}.
Thus $e=t$.
If $t>0$, the complete quiver is a vortex, with source $3$ and directed triangle $1\to2\to4\to1$ \claim{B22.bD-vortex}, and if $t=0$, $C$ is acyclic \claim{B22.bD-t0}.
Both are impossible.

Finally suppose $f\le0$.
Now $p=f+Te+t$ \claim{B22.bD-fneg-pf}, so an integer $e\ge2$ would give $p\ge T$ \claim{B22.bD-fneg-e2}, and $e\le1$.
If $e\le0$, or if $e=1$ and $t=f=0$, then $C$ is acyclic \claim{B22.bD-fneg-e0, B22.bD-fneg-e1-00}.
Otherwise $e=1$, and $C$ has a single proper wiggle class, containing $\corder{1,2,4,3}$ \claim{B22.bD-order-tf, B22.bD-order-t0, B22.bD-order-f0}.
Evaluation gives $m=T^2+T+3+t(t-1)+f^2+f(T+1-Tt)>T$ \claim{B22.bD-m, B22.bD-mbig}.
This excludes the whole case $x<0$, $d\le0$.

\subsubsection{Case $x\ge0$}\label{app:case3}
The first mutation can reach height $F$ only if $c=-t<0$, where $1\le t\le T$.
Write $w=f-xt$ and
\[
    z=w+[-d]_+[e-t]_+-[d]_+[t-e]_+.
\]
In each of the eight sign cases of $d$, $e-t$ and $z$, which are the claims \claim{B23.paths-+++} through \claim{B23.paths----},
\[
    \qtuple{P}=(-1,-x,t,d,e-t,w),\qquad\qtuple{Q}=(1,-x-[-d]_+,\min(t,e),-d,t-e,z),
\]
and the entries we need after the third mutation are
\begin{align*}
    b_{12}(\mu_3Q)&=1-(x+[-d]_+)[-d]_+,\\
    b_{14}(\mu_3Q)&=\min(t,e)-(x+[-d]_+)[-z]_+,\\
    b_{24}(\mu_3Q)&=t-e+[-d]_+[z]_+-[d]_+[-z]_+.
\end{align*}
The Pfaffian is $p=f-xe-td$ \claim{B23.pf}.
Since every starting entry has absolute value at most $T$, the first high quiver must satisfy $e-t=-F$ or $w=-F$.

First suppose $e-t=-F$.
Put $u=F-t\ge1$, so $e=-u$.
Then $z=w-[d]_+F$, and the third $b_{14}$ bound reads
\begin{equation}\label{eq:nonneg-first-bound}
    (x+[-d]_+)[-z]_+\le t.
\end{equation}
If $d=-D<0$, the third $b_{24}$ entry is $F+D[w]_+$, so $w\le0$ \claim{B23.eF-dneg-w}.
Set $v=-w\ge0$.
Equation~\eqref{eq:nonneg-first-bound} gives $(x+D)v\le t$, and then $x\ge1$ would give $p=xF+tD-v\ge F$ \claim{B23.eF-dneg-pf, B23.eF-dneg-x}.
Thus $x=0$ and $f=w\le0$, and $C$ is acyclic \claim{B23.eF-acyc1}.

If $d\ge0$, the second height bound gives $w-dF\ge-F$, which forces $d\le1$ since $w\le T$ \claim{B23.eF-dpos-d}.
For $d=0$, equation~\eqref{eq:nonneg-first-bound} is $x[-w]_+\le t$, and $p=w+xF$ \claim{B23.eF-d0-pf}.
An $x\ge2$ is impossible, since $p\ge2F$ if $w\ge0$, and $p\ge2F-t>F$ if $w<0$ \claim{B23.eF-d0-x2a, B23.eF-d0-x2b}.
When $x=1$, the same inequality gives $f=t+w\ge0$ \claim{B23.eF-d0-x1}.
For $d=1$, the second height bound gives $w\ge0$, and $x\ge1$ is impossible \claim{B23.eF-d1-x}.
Hence $x=0$ and $f=w\ge0$.
The case $x=d=0$ makes $C$ acyclic \claim{B23.eF-acyc2}.
Thus only $(x,d)=(0,1)$ or $(1,0)$ remains, with $f>0$, because $f=0$ is again acyclic \claim{B23.eF-acyc3, B23.eF-acyc4}.
In both cases $C$ has a single proper wiggle class, containing $\omega_0=\corder{1,2,3,4}$ \claim{B23.eF-ord01, B23.eF-ord10}, and
\[
    m=2+t^2+u^2+tu+f^2-\lambda f,\qquad\lambda=\begin{cases}
        F,&(x,d)=(0,1),\\
        t,&(x,d)=(1,0),
        \end{cases}
\]
which is at least $F$ in both cases, contradicting $m\le T$, by \claim{B23.eF-m01, B23.eF-m01-big} and \claim{B23.eF-m10, B23.eF-m10-big}.
This excludes $e-t=-F$.

We may now suppose $w=-F$ and $e-t>-F$.
Since $f=xt-F>-F$, we have $x\ge1$ \claim{B23.wF-x}.
Also $|e-t|<F$ and $|\min(t,e)|<F$.
For $d=-D<0$,
\[
    z=-F+D[e-t]_+,
\]
and the second quiver can be high only if $x+D=F$ or $z=\pm F$.
If $x+D=F$, the third $b_{12}$ bound gives $D=1$ \claim{B23.wF-dneg-D}, the third $b_{14}$ bound then gives $z\ge-1$ \claim{B23.wF-dneg-z}, and $z=-F+[e-t]_+\ge-1$ is impossible because $e-t\le F-2$ \claim{B23.wF-dneg-et}.
If $z=-F$, the third $b_{14}$ entry is less than $-F$ \claim{B23.wF-dneg-zmF}.
If $z=F$, then $e-t>0$, the third $b_{24}$ bound gives $D=1$ \claim{B23.wF-dneg-zF-D}, and $z=-F+(e-t)=F$ is impossible \claim{B23.wF-dneg-zF}.
All possibilities with $d<0$ are excluded.

For $d\ge0$, we have $z=-F-d[t-e]_+\le-F$.
The second height bound therefore gives $z=-F$ and $d[t-e]_+=0$ \claim{B23.wF-dpos-z}.
The third $b_{14}$ bound $\min(t,e)-xF\ge-F$ forces $x=1$ and $e\ge0$ \claim{B23.wF-dpos-x, B23.wF-dpos-e}.
If $d>0$, then $e\ge t$, and the third $b_{24}$ bound forces $d=1$ and $e=t$ \claim{B23.wF-dpos-d, B23.wF-dpos-e2}.
But then $p=-F-t$, contrary to~\eqref{eq:prefixscalar} \claim{B23.wF-d1-pf, B23.wF-d1-p}.

Only $d=0$ remains.
Put $u=F-t$, so $\qtuple{C}=(1,1,-t,0,e,-u)$ with $t,u\ge1$ and $e\ge0$.
If $e=0$, $C$ is acyclic \claim{B23.wF-d0-acyc}.
For $e>0$, the unique proper wiggle class contains $\corder{1,2,3,4}$ \claim{B23.wF-d0-order}, and evaluation gives $m=2+t^2+u^2+e^2+tu-te>F$ \claim{B23.wF-d0-m, B23.wF-d0-big}.
This final contradiction excludes every case with $x\ge0$, including
zero entries.
Together with Appendices~\ref{app:case1} and~\ref{app:case2}, this shows that $C$ is $A_c$ or $I_{t,u}$.

\subsection{The two remaining cases}\label{app:equal}
For $A_c$, with $\qtuple{A_c}=(1,-2,c,1,0,F-2c)$ and $1\le c\le T$, the quivers after $\mword{1}$, $\mword{1,2}$, $\mword{1,2,3}$ and $\mword{1,2,3,1}$ are high, and the quiver after $\mword{1,2,3,1,3}$ is low \claim{B3.A-segment, B3.A-low-lo0, B3.A-low-hi0, B3.A-low-lo1, B3.A-low-hi1}.
After $\mword{1,2,3}$, mutations at $2,4$ create entries $2F-c>F$ and $F^2-1>F$, respectively \claim{B3.A-forced1, B3.A-forced1b, B3.A-forced-2Fc, B3.A-forced-F2}.
Repeating the last mutation would cancel, so the next vertex must be $1$.
After $\mword{1,2,3,1}$, mutations at $2,4$ create entries $2F-c>F$ and $-1-F(F-c)<-F$, so the next vertex must be $3$ \claim{B3.A-forced2, B3.A-forced2b, B3.A-forced-FFc}.
Thus the first return is exactly $\mword{1,2,3,1,3}$.
All three quivers along $\mword{2,1,3}$ are low by the same claims \claim{B3.A-replacement}.
For the relabeling $\pi_{123}$ \claim{B3.A-relabel},
\[
    \mu_{\mword{2,1,3}}A_c=(\mu_{\mword{1,2,3,1,3}}A_c)^{\pi_{123}}.
\]
This replacement removes four high quivers, and the tail is transported by \eqref{eq:tailtransport}.

For $I_{t,u}$, with $\qtuple{I_{t,u}}=(1,-1,-t,2,t-F,F-u)$, exactly the quivers after $\mword{1}$ and $\mword{1,2}$ are high, and the quiver after $\mword{1,2,3}$ is low \claim{B3.I-segment}, so the first return is $\mword{1,2,3}$.
Both replacement mutation sequences terminate at the quiver $J$ with $\qtuple{J}=(1,-2,F-u,1,u,t-F)=\qtuple{(\mu_{\mword{1,2,3}}I_{t,u})^{\pi_{123}}}$ \claim{B3.I-endpoint, B3.I-tgu, B3.I-ugt}.
Every quiver along them is low, using $T\ge3$ \claim{B3.I-low-tgu-lo0, \dots, B3.I-low-ugt-hi5}.
Thus both replacement paths, including their common endpoint $J$, are entirely low and each removes two high quivers.

Finally, suppose $t=u$.
The Pfaffian of $I_{t,t}$ is $p=-2t$ \claim{B3.I-tt-pf}.
Thus \eqref{eq:prefixscalar} gives $4t<m\le F-1$, and in particular $k=F-t\ge2$ \claim{B3.I-tt-k}.
The full triangle on $\{2,3,4\}$ is an oriented cycle with multiplicities $(2,k,k)$ \claim{B3.I-tt-triangle}, whose Markov constant is $2^2+k^2+k^2-2k^2=4$ \claim{B3.I-tt-markov}.
By Proposition~\ref{prop:rank3}, this triangle is mutation-cyclic, contradicting mutation-acyclicity of the whole quiver.
So $t\ne u$, and Proposition~\ref{prop:unit} is proved. \qed

\section*{Use of AI}

After being prompted at length by the authors, the main results were discovered by GPT-Astra (OpenAI) and the original statements were formalized in Lean (\url{https://github.com/KTKAW/rank4-paper}).
The results follow from many well-known results in the literature, along with the AI-discovery of the family $Q(t)$ and many case checks carried out by AI systems.
The exact checks in all of the appendix calculations are also verified by Python scripts.
The text in this article was entirely organized and reviewed by humans. 
Significant portions of the text were generated, reviewed, or edited by AI.
There was significant effort put in by the human team to generalize the results as far as possible, to simplify the proofs, to give appropriate credit to the human results that came before, and to bring the exposition up to standard.

\bibliographystyle{amsalpha}
\bibliography{references}
\end{document}